\documentclass[reqno,english]{amsart}
\usepackage[T1]{fontenc}
\usepackage[latin9]{inputenc}
\usepackage{geometry}
\usepackage{babel}
\usepackage{array}
\usepackage{refstyle}
\usepackage{float}
\usepackage{mathrsfs}
\usepackage{mathtools}
\usepackage{enumitem}
\usepackage{multirow}
\usepackage{amstext}
\usepackage{amsthm}
\usepackage{amssymb}
\usepackage{graphicx}
\usepackage[all]{xy}
\PassOptionsToPackage{normalem}{ulem}
\usepackage{ulem}
\usepackage[unicode=true,pdfusetitle,
 bookmarks=true,bookmarksnumbered=false,bookmarksopen=false,
 breaklinks=false,pdfborder={0 0 0},pdfborderstyle={},backref=false,colorlinks=false]
 {hyperref}
\hypersetup{
 colorlinks=true,citecolor=blue,linkcolor=blue,linktocpage=true}

\makeatletter

\AtBeginDocument{\providecommand\secref[1]{\ref{sec:#1}}}
\AtBeginDocument{\providecommand\defref[1]{\ref{def:#1}}}
\AtBeginDocument{\providecommand\lemref[1]{\ref{lem:#1}}}
\AtBeginDocument{\providecommand\thmref[1]{\ref{thm:#1}}}
\AtBeginDocument{\providecommand\corref[1]{\ref{cor:#1}}}
\AtBeginDocument{\providecommand\remref[1]{\ref{rem:#1}}}
\AtBeginDocument{\providecommand\figref[1]{\ref{fig:#1}}}
\AtBeginDocument{\providecommand\exaref[1]{\ref{exa:#1}}}
\AtBeginDocument{\providecommand\subsecref[1]{\ref{subsec:#1}}}
\newcommand{\lyxmathsym}[1]{\ifmmode\begingroup\def\b@ld{bold}
  \text{\ifx\math@version\b@ld\bfseries\fi#1}\endgroup\else#1\fi}

\providecommand{\tabularnewline}{\\}
\RS@ifundefined{subsecref}
  {\newref{subsec}{name = \RSsectxt}}
  {}
\RS@ifundefined{thmref}
  {\def\RSthmtxt{theorem~}\newref{thm}{name = \RSthmtxt}}
  {}
\RS@ifundefined{lemref}
  {\def\RSlemtxt{lemma~}\newref{lem}{name = \RSlemtxt}}
  {}

\numberwithin{equation}{section}
\numberwithin{figure}{section}
\numberwithin{table}{section}
\theoremstyle{plain}
\newtheorem{thm}{\protect\theoremname}[section]
\theoremstyle{definition}
\newtheorem{defn}[thm]{\protect\definitionname}
\theoremstyle{plain}
\newtheorem{lem}[thm]{\protect\lemmaname}
\theoremstyle{remark}
\newtheorem{rem}[thm]{\protect\remarkname}
\theoremstyle{plain}
\newtheorem{cor}[thm]{\protect\corollaryname}
\theoremstyle{definition}
\newtheorem{example}[thm]{\protect\examplename}
\theoremstyle{plain}
\newtheorem{prop}[thm]{\protect\propositionname}

\providecommand{\MR}[1]{}

\usepackage{bbm}
\allowdisplaybreaks
\usepackage{needspace}
\usepackage{refstyle}

\newref{lem}{refcmd={Lemma \ref{#1}}}
\newref{thm}{refcmd={Theorem \ref{#1}}}
\newref{cor}{refcmd={Corollary \ref{#1}}}
\newref{sec}{refcmd={Section \ref{#1}}}
\newref{subsec}{refcmd={Section \ref{#1}}}
\newref{chap}{refcmd={Chapter \ref{#1}}}
\newref{prop}{refcmd={Proposition \ref{#1}}}
\newref{exa}{refcmd={Example \ref{#1}}}
\newref{tab}{refcmd={Table \ref{#1}}}
\newref{rem}{refcmd={Remark \ref{#1}}}
\newref{def}{refcmd={Definition \ref{#1}}}
\newref{fig}{refcmd={Figure \ref{#1}}}
\newref{claim}{refcmd={Claim \ref{#1}}}

\@ifundefined{showcaptionsetup}{}{%
 \PassOptionsToPackage{caption=false}{subfig}}
\usepackage{subfig}
\AtBeginDocument{
  
}

\makeatother

\providecommand{\corollaryname}{Corollary}
\providecommand{\definitionname}{Definition}
\providecommand{\examplename}{Example}
\providecommand{\lemmaname}{Lemma}
\providecommand{\propositionname}{Proposition}
\providecommand{\remarkname}{Remark}
\providecommand{\theoremname}{Theorem}

\begin{document}
\title[]{Dual pairs of operators, harmonic analysis of singular non-atomic
measures and Krein-Feller diffusion}
\author{Palle E.T. Jorgensen}
\address{(Palle E.T. Jorgensen) Department of Mathematics, The University of
Iowa, Iowa City, IA 52242-1419, U.S.A.}
\email{palle-jorgensen@uiowa.edu}
\author{James Tian}
\address{(James F. Tian) Mathematical Reviews, 416 4th Street Ann Arbor, MI
48103-4816, U.S.A.}
\email{jft@ams.org}
\begin{abstract}
We show that a Krein-Feller operator is naturally associated to a
fixed measure $\mu$, assumed positive, $\sigma$-finite, and non-atomic.
Dual pairs of operators are introduced, carried by the two Hilbert
spaces, $L^{2}\left(\mu\right)$ and $L^{2}\left(\lambda\right)$,
where $\lambda$ denotes Lebesgue measure. An associated operator
pair consists of two specific densely defined (unbounded) operators,
each one contained in the adjoint of the other. This then yields a
rigorous analysis of the corresponding $\mu$-Krein-Feller operator
as a closable quadratic form. As an application, for a given measure
$\mu$, including the case of fractal measures, we compute the associated
diffusion, semigroup, Dirichlet forms, and $\mu$-generalized heat
equation.
\end{abstract}

\subjclass[2000]{Primary: 47B32, 47B25, 47E05, 46N20. Secondary: 46E22, 46N30, 46N50,
60G15.}
\keywords{Hilbert space, reproducing kernel Hilbert space, harmonic analysis,
Gaussian free fields, transforms, covariance, generalized Ito integration,
Krein-Feller operators, unbounded operators, dual pairs, semigroup,
diffusion, Ito lemma, Cantor measures, iterated function systems,
fractals, selfadjoint extensions, generalized Brownian motion, Dirichlet
forms.}
\maketitle

\section{Introduction}

Recently there have been several advances to an harmonic analysis
of Krein-Feller operators for classes of singular measures. (Intuitively,
a Krein-Feller operator is an analogue of a Laplacian for classical
domains, as they arise in diffusion problems and in potential theory.)

In fact there are recent papers which cover the theory from the point
of fractal analysis, see e.g., \cite{MR2407555,MR2017701,MR1204707,MR4184588,MR4025934,MR3110587};
as well as applications to physics and to signal processing, e.g.,
the papers \cite{MR2213587,MR1408681,MR885633,MR849517,MR2973393,MR1660800,MR933819}.
Our present approach to Krein-Feller operators is motivated by both
of these new trends; but our approach is based on a new duality. It
combines a new transform theory based on the theory of reproducing
kernel Hilbert spaces (RKHSs), and a new technology introduced here,
based on pairs of unbounded densely defined operators, each one contained
in the adjoint of the other.

A word about the terminology \textquotedblleft Krein-Feller operator\textquotedblright{}
(details are cited inside the paper): Mark Krein has pioneered a number
of powerful Hilbert space-based tools which have found numerous applications,
and the present problem is a case in point. Krein's operator theory
(cited below) forms the foundation in our approach to problems for
unbounded operators with dense domain in Hilbert space. Sections \ref{sec:DP}
and \ref{sec:rkhs} below will elaborate on this. William Feller,
in the name \textquotedblleft Krein-Feller operator\textquotedblright{}
refers to the role of the KF-operator in the study of diffusion. Indeed,
W. Feller was one of the pioneers in our understanding of diffusion,
diffusion-semigroups, and their analysis. Hence later authors have
adopted the name \textquotedblleft Krein-Feller operators\textquotedblright{}
for the associated semigroup generators. There are interesting connections
to inverse problems, and prediction theory, see \cite{MR0448523}.
A nice presentation of this, and early work of Krein and Feller, is
\cite[chapter 5]{MR0448523}. We shall include additional details
on this point in \secref{KF} below.

Starting with a fixed positive non-atomic Borel measure $\mu$ (with
support contained in $\mathbb{R}$), then, informally, the associated
\emph{Krein-Feller operator} (denoted $K_{F}=K_{F}^{\left(\mu\right)}$)
is $K_{F}=\frac{d}{d\mu}\frac{d}{dx}$. The meaning of ``$d/d\mu$''
will be made precise. If $x>0$, set $g_{\mu}\left(x\right)=\mu\left(\left[0,x\right]\right)$,
i.e., the cumulative distribution. For $\psi\in C^{1}$, we have $\frac{d}{d\mu}\left(\psi\circ g_{\mu}\right)=\psi'\circ g_{\mu}$.
A key step in our consideration is a rigorous study of $K_{F}$ as
an unbounded (symmetric) operator in $L^{2}\left(\mu\right)$.

\textbf{Organization}: We begin in \secref{DP} with the framework
for our dual pair analysis. This is presented in the rather general
setting of pairs of Hilbert spaces, and associated pairs of densely
defined (unbounded) operators. Particular choices of dual pairs of
operators are then applied to a rigorous analysis of Krein-Feller
operators in \secref{KF}. The framework for these considerations
is a fixed measure $\mu$, assumed positive, sigma-finite, and non-atomic.
Hence, our starting point is a specified and fixed measure $\mu$
(generally singular, e.g., a Cantor measure). The two Hilbert spaces
for the corresponding dual pair of unbounded operators will then be
$L^{2}(\mu)$ and $L^{2}(\lambda)$, where $\lambda$ denotes Lebesgue
measure, or its restriction to a chosen interval. The rest of \secref{KF}
will deal with an analysis of the associated diffusion, Markov process,
semigroup, and a corresponding $\mu$-generalized heat equation. \secref{class}
studies Stieltjes measures $df$ globally. For this purpose, we introduce
a Hilbert space $\mathscr{H}_{\text{class}}$ of \textquotedblleft sigma
functions\textquotedblright{} as a Hilbert space of certain equivalence
classes. Starting with a fixed Stieltjes measure $df$, we then identify
its pairwise mutually singular components with corresponding orthogonal
\textquotedblleft pieces\textquotedblright{} in the Hilbert space
$\mathscr{H}_{\text{class}}$.

In a general framework, these settings correspond to suitably specified
Dirichlet forms; the subject of \secref{rkhs}. An application to
iterated function system (IFS) measures is also included in \secref{rkhs}.
Our analysis of specific applications relies on several new tools;
one in particular derives from consideration of associated reproducing
kernel Hilbert spaces (RKHSs), and Gaussian fields. This is covered
in \secref{sc}.

\section{\label{sec:DP}Dual pairs of operators in Hilbert space}

The notion of \emph{dual pairs} we shall need here is in \defref{sp}
below. But the operators in question will act between two Hilbert
spaces to be specified. Hence, we shall first need to recall some
properties of \emph{unbounded operators} with specified dense domains;
especially the precise definition (Def \ref{def:B1}) of the adjoints
of such operators. With this accomplished, the dual pair definition
(Def \ref{def:sp}) for a pair of densely defined operators amounts
to the assertion that each operator in the pair be contained in the
dual of the other (\lemref{ST}). We shall need this in our analysis
of classes of \emph{Krein-Feller operators} introduced in \secref{KF}
below. For a given measure $\mu$ and associated Krein-Feller operator
$K_{F}$, we shall then identify a dual pair which provides a factorization
of this Krein-Feller operator $K_{F}$. Of course, $K_{F}$ is an
unbounded operator, symmetric and semibounded; so our dual pair factorization
will present us with a canonical selfadjoint extension, see \lemref{ST}.
Background references for this include \cite{MR0231220,MR0442564,MR1009163,MR65809,MR1335452,MR2953553}.

Let $\mathscr{H}_{1}$ and $\mathscr{H}_{2}$ be complex Hilbert spaces.
If $\mathscr{H}_{1}\xrightarrow{\;T\;}\mathscr{H}_{2}$ represents
a linear operator from $\mathscr{H}_{1}$ into $\mathscr{H}_{2}$,
we shall denote 
\begin{equation}
dom\left(T\right)=\left\{ \varphi\in\mathscr{H}_{1}\mid\mbox{\ensuremath{T\varphi} is well-defined}\right\} ,\label{eq:in1}
\end{equation}
the domain of $T$, and 
\begin{equation}
ran\left(T\right)=\left\{ T\varphi\mid\varphi\in dom\left(T\right)\right\} ,\label{eq:in2}
\end{equation}
the range of $T$. The closure of $ran\left(T\right)$ will be denoted
$\overline{ran\left(T\right)}$.
\begin{defn}
\label{def:B1}Let $T:\mathscr{H}_{1}\rightarrow\mathscr{H}_{2}$
be a densely defined operator, and let
\begin{align}
dom(T^{*})= & \Big\{ h_{2}\in\mathscr{H}_{2}\mid\mbox{\ensuremath{\exists C=C_{h_{2}}<\infty,} s.t. \ensuremath{\left|\left\langle h_{2},T\varphi\right\rangle _{2}\right|\leq C\left\Vert \varphi\right\Vert _{1}}}\nonumber \\
 & \quad\mbox{holds for \ensuremath{\forall\varphi\in dom\left(T\right)}}\Big\}.\label{eq:in3}
\end{align}
By Riesz' theorem, there is a unique $\eta\in\mathscr{H}_{1}$ for
which
\begin{equation}
\left\langle \eta,\varphi\right\rangle _{1}=\left\langle h_{2},T\varphi\right\rangle _{2},\quad h_{2}\in dom(T^{*}),\;\varphi\in dom\left(T\right),\label{eq:4}
\end{equation}
and the adjoint operator is defined as $T^{*}h_{2}=\eta$. See the
diagram below:

\[
\xymatrix{\mathscr{H}_{1}\ar@/^{1pc}/[rr]^{T} &  & \mathscr{H}_{2}\ar@/^{1pc}/[ll]^{T^{*}}}
\]
\end{defn}

\begin{defn}
The \emph{graph} of $T:\mathscr{H}_{1}\rightarrow\mathscr{H}_{2}$
is 
\begin{equation}
G_{T}:=\left\{ \begin{bmatrix}\varphi\\
T\varphi
\end{bmatrix}\mid\varphi\in dom\left(T\right)\right\} \subset\text{\ensuremath{\mathscr{H}}}_{1}\oplus\text{\ensuremath{\mathscr{H}}}_{2},\label{eq:in6}
\end{equation}
where $\mathscr{H}_{1}\oplus\mathscr{H}_{2}$ is a Hilbert space under
the natural inner product
\begin{equation}
\left\langle \begin{bmatrix}\varphi_{1}\\
\varphi_{2}
\end{bmatrix},\begin{bmatrix}\psi_{1}\\
\psi_{2}
\end{bmatrix}\right\rangle :=\left\langle \varphi_{1},\psi_{1}\right\rangle _{\mathscr{H}_{1}}+\left\langle \varphi_{2},\psi_{2}\right\rangle _{\mathscr{H}_{2}}.
\end{equation}
\end{defn}

\begin{defn}
Let $T:\mathscr{H}_{1}\rightarrow\mathscr{H}_{2}$ be a linear operator. 
\begin{enumerate}
\item \label{enu:g1}$T$ is \emph{closed} if $G_{T}$ is closed in $\text{\ensuremath{\mathscr{H}}}_{1}\oplus\text{\ensuremath{\mathscr{H}}}_{2}$.
\item \label{enu:g2} $T$ is \emph{closable} if $\overline{\ensuremath{G_{T}}}$
is the graph of an operator. 
\item If (\ref{enu:g2}) holds, the operator corresponding to $\overline{G_{T}}$,
denoted $\overline{T}$, is called the \emph{closure}, i.e., 
\begin{equation}
\overline{G_{T}}=G_{\overline{T}}.\label{eq:in7}
\end{equation}
\end{enumerate}
\end{defn}

We shall need the following two results for unbounded operators, see
e.g., \cite{MR1009163,MR2953553,MR1157815}. To clarify notation,
and for the benefit of the reader, we have included them below in
the form they are needed.
\begin{thm}
\label{thm:in1}Let $T:\mathscr{H}_{1}\rightarrow\mathscr{H}_{2}$
be a densely defined operator. Then
\begin{enumerate}
\item \label{enu:b1}$T^{*}$ is closed;
\item \label{enu:b2}$T$ is closable $\Longleftrightarrow$ $dom\left(T^{*}\right)$
is dense;
\item \label{enu:b3}$T$ is closable $\Longrightarrow$ $(\overline{T})^{*}=T^{*}$. 
\end{enumerate}
\end{thm}

\begin{thm}[von Neumann, polar decomposition/factorization, \cite{MR1009163}]
\label{thm:vN}Let $\mathscr{H}_{i}$, $i=1,2$, be two Hilbert spaces,
and let $T$ be a closed operator from $\mathscr{H}_{1}$ into $\mathscr{H}_{2}$
having dense domain in $\mathscr{H}_{1}$; then $T^{*}T$ is selfadjoint
in $\mathscr{H}_{1}$, $TT^{*}$ is selfadjoint in $\mathscr{H}_{2}$,
both with dense domains.

Moreover, there is a partial isometry $J:\mathscr{H}_{1}\rightarrow\mathscr{H}_{2}$
such that 
\begin{equation}
T=J\left(T^{*}T\right)^{\frac{1}{2}}=\left(TT^{*}\right)^{\frac{1}{2}}J\label{eq:vn1}
\end{equation}
holds on $dom\left(T\right)$. (Equation (\ref{eq:vn1}) is called
the polar decomposition of $T$.)
\end{thm}

\begin{defn}[symmetric pair]
\label{def:sp} For $i=1,2$, let $\mathscr{H}_{i}$ be two Hilbert
spaces, and suppose $\mathscr{D}_{i}\subset\mathscr{H}_{i}$ are given
dense subspaces. 

We say that a pair of operators $\left(S,T\right)$ forms a \emph{symmetric
pair} if $dom\left(T\right)=\mathscr{D}_{1}$, and $dom\left(S\right)=\mathscr{D}_{2}$;
and moreover, 
\begin{equation}
\left\langle Tu,v\right\rangle _{\mathscr{H}_{2}}=\left\langle u,Sv\right\rangle _{\mathscr{H}_{1}}\label{eq:sp1}
\end{equation}
holds for $\forall u\in\mathscr{D}_{1}$, $\forall v\in\mathscr{D}_{2}$.
See the diagram below: 
\[
\xymatrix{\mathscr{H}_{1}\ar@/^{1.2pc}/[rr]^{T} &  & \mathscr{H}_{2}\ar@/^{1.2pc}/[ll]^{S}}
\]
\end{defn}

\begin{lem}[Dual Pair \cite{MR3552934}]
\label{lem:ST}Let $\left(S,T\right)$ be the pair of operators specified
in (\ref{eq:sp1}). Then, we have
\begin{equation}
T\subset S^{*},\quad S\subset T^{*}\label{eq:B10}
\end{equation}
(containment of graphs.) Moreover, the two operators $S^{*}\overline{S}$
and $T^{*}\overline{T}$ are selfadjoint.
\end{lem}

It is immediate from (\ref{eq:B10}) that both $S$ and $T$ are \emph{closable}. 
\begin{defn}
We say that a symmetric pair is \emph{maximal} if 
\[
\overline{T}=S^{*},\quad\text{and}\quad\overline{S}=T^{*}.
\]
\end{defn}

With the starting point, a fixed positive non-atomic Borel measure
$\mu$ with support on an interval, we now show how the operator theoretic
framework of dual pairs (\defref{sp}) offers an explicit setting
for the study of spectral theory of the associated Krein-Feller operator.
In particular, we show in the subsequent sections how key features
of our dual pair framework from the discussion above serves to yield
explicit new results for the Krein-Feller operator, for example \thmref{D2},
\lemref{C9}, \corref{D3}.

\section{\label{sec:KF}Krein-Feller operators, and their properties }

For a given measure $\mu$ we shall offer several tools in our analysis
of the associated Krein-Feller operator $K_{F}$. One will make use
of an appropriate dual pair of operators (see Sec \ref{sec:DP}, and
\thmref{D2} and \corref{D3} below). The other is more direct; it
is sketched in the present section. In \thmref{C1}, we present the
inverse of $K_{F}$ as an explicit integral operator. This will be
especially useful in our analysis of the spectrum of diverse \emph{selfadjoint
extensions} of $K_{F}$. Background references for this include \cite{MR2966130,MR2793121,MR2384473,MR562914,MR0345224}.
For basics on Stieltjes measures, fractals, and transformation rules
for measures, readers may wish to consult \cite{MR3616046,MR0442564,MR625600,MR735967,MR0214150,MR0047744,MR1157815,MR2558684,MR3275999}.

Perhaps it is appropriate to add a comment on the role of W. Feller,
in the name \textquotedblleft Krein-Feller operator.\textquotedblright{}
Feller was one of the pioneers in the study of diffusion, diffusion-semigroups,
and their analysis. Hence later authors have adopted the name \textquotedblleft Krein-Feller
operators\textquotedblright{} for the associated semigroup generators.
We shall elaborate this point in the next section. A list of references
which covers this viewpoint is long, but it includes the following,
\cite{MR65809,MR87254,MR234524}.

\textbf{Terminology convention.} Fixing a measure $\mu$ as specified,
then formally, the notation $\nabla_{\mu}$ (see (\ref{eq:C2})) and
$T_{\mu}$ stand for the same operation, but in the theorem below,
we are referring to a specific pair of Hilbert spaces, and the notation
$T_{\mu}$ is used to stress this point. The conclusion of \thmref{D2}
is that the Krein-Feller operator $K_{F}$ then has a symmetric dual-pair
realization in the sense of \defref{sp}.

If $f$ is a function on $\mathbb{R}$ (or defined on a subinterval),
assumed to be locally of bounded variation, then we shall denote by
$df$ the corresponding Stieltjes measure. (Recall $df$ is defined
first on intervals $(x,y]$ by $df\left((x,y]\right):=f\left(y\right)-f\left(x\right)$,
and then extended to the Borel $\sigma$-algebra $\mathscr{B}$ by
the usual $\sigma$-algebra-completion procedure.) If $\mu$ is a
fixed positive Borel measure, we then consider the corresponding Radon-Nikodym
derivative, denoted 
\begin{equation}
f^{\left(\mu\right)}=\nabla^{\left(\mu\right)}f=df/d\mu.
\end{equation}
It is determined by, 
\begin{equation}
f\left(y\right)-f\left(x\right)=\int_{x}^{y}\left(\nabla^{\left(\mu\right)}f\right)d\mu;\label{eq:rC2}
\end{equation}
abbreviated $\left(\nabla^{\left(\mu\right)}f\right)d\mu=df$.

The the Krein-Feller operator $K_{F}$ is defined as 
\begin{equation}
K_{F}=\frac{d}{d\mu}\frac{d}{dx}=\nabla_{\mu}\frac{d}{dx}.
\end{equation}

In what follows, we denote by $J$ the unit interval $\left[0,1\right]$. 
\begin{thm}[A symmetric pair for $\nabla_{\mu}$]
\label{thm:D2}If $\varphi\in C_{c}^{\infty}\left(J\right)$ then
\begin{equation}
-\int\varphi'f\,dx=\int_{J}\varphi\left(T_{\mu}f\right)d\mu;\label{eq:cc29}
\end{equation}
so we obtain the dual pair of operators: 
\[
\xymatrix{\mathscr{L}^{2}\left(\mu\right)\ar@/^{1.3pc}/[rr]^{T_{\mu}} &  & L^{2}\left(\mu\right)\ar@/^{1.3pc}/[ll]^{D=-\frac{d}{dx}}}
\]
Here, 
\begin{equation}
\mathscr{L}^{2}\left(\mu\right):=\overline{dom\left(T_{\mu}\right)}^{L^{2}\left(\lambda\right)},
\end{equation}
i.e., the $L^{2}\left(\lambda\right)$-closure of $dom\left(T_{\mu}\right)$;
see (\ref{eq:Dmax}).
\end{thm}

\begin{proof}
One checks that 
\begin{equation}
\int\varphi\,df=-\int\varphi'f\,dx\label{eq:D4}
\end{equation}
holds for all $\varphi\in C_{c}^{\infty}\left(J\right)$, using integration
by parts. 

Details: Let $f$ and $\varphi$ be as specified, $\varphi\in C_{c}^{1}\left(J\right)$,
$f$ locally bounded variation s.t. $f^{\left(\mu\right)}=T_{\mu}f\in L_{loc}^{2}\left(\mu\right)$
is well defined. For the integral $\int_{J}\varphi\left(T_{\mu}f\right)d\mu$,
we therefore get the following approximation via choices of partitions
in the interval $J$: $x_{0}<x_{1}<\cdots$: 
\begin{eqnarray*}
\int_{J}\varphi\left(T_{\mu}f\right)d\mu & \simeq & \sum_{i}\int_{x_{i}}^{x_{i+1}}\varphi\left(T_{\mu}f\right)d\mu\\
 & \simeq & \sum_{i}\varphi\left(x_{i}\right)\int_{x_{i}}^{x_{i+1}}f^{\left(\mu\right)}d\mu\\
 & \underset{\left(\text{\ref{eq:rC2}}\right)}{=} & \sum_{i}\varphi\left(x_{i}\right)\underset{\text{the Stieltjes measure \ensuremath{df}}}{\underbrace{df\text{\ensuremath{\left(\left[x_{i},x_{i+1}\right]\right)}}}}\\
 & \simeq & \int\varphi df=-\int_{J}\varphi'\left(x\right)f\left(x\right)dx.
\end{eqnarray*}
 
\end{proof}

\subsection{\label{subsec:tl}Realization of $T_{\mu}$ as a skew-symmetric operator
with dense domain in $L^{2}\left(\mu\right)$}

Fix a non-atomic measure $\mu$ on $\left[0,1\right]$. Let 
\begin{equation}
\mathscr{D}_{1}:=\left\{ f:f\left(x\right)=f\left(0\right)+\int_{0}^{x}f^{\left(\mu\right)}d\mu,\;f^{\left(\mu\right)}\in L^{2}\left(\mu\right),\:\text{for all \ensuremath{x}}\right\} .\label{eq:Dmax}
\end{equation}
Then $\mathscr{D}_{1}\subset L^{2}\left(\mu\right)\cap C\left(\left[0,1\right]\right)$. 

Define 
\begin{equation}
\nabla_{\mu}f=f^{\left(\mu\right)},\quad\forall f\in\mathscr{D}_{1}.\label{eq:C2}
\end{equation}

In the lemma below we express eq (\ref{eq:C2}) for the operator $\nabla_{\mu}$
(acting on functions $f$) in terms of associated Stieltjes measures.
This point is summarized best in eq (\ref{eq:ca3}) in \lemref{C1}
below, where $df$ then denotes the \emph{Stieltjes measure} corresponding
to some function $f$. In the sequel we shall reserve the notation
$df$ for Stieltjes measure, (not to be confused with notions of differential.)
Recall that for the Stieltjes measure $df$ to make sense, the function
$f$ must be assumed to be locally of bounded variation.
\begin{lem}
\label{lem:C1}Let $f$ be a function on $\mathbb{R}$, assumed to
be locally of bounded variation, so that the Stieltjes measure $df$
is well defined. Let $\mu$ be a positive measure defined on the Borel
$\sigma$-algebra $\mathscr{B}$, and assume that 
\begin{equation}
df\ll\mu,\label{eq:ca1}
\end{equation}
i.e., that the implication (\ref{eq:ca2}) below holds: 
\begin{equation}
\mu\left(B\right)=0\Longrightarrow df\left(B\right)=0.\label{eq:ca2}
\end{equation}
Let $f^{\left(\mu\right)}$ be the corresponding Radon-Nikodym derivative
(also denoted $f^{\left(\mu\right)}=\nabla_{\mu}f$), then 
\begin{equation}
df=f^{\left(\mu\right)}d\mu.\label{eq:ca3}
\end{equation}
\end{lem}

\begin{proof}
The assertion in (\ref{eq:ca3}) amounts to the identity 
\begin{equation}
df\left(B\right)=\int_{B}f^{\left(\mu\right)}d\mu,\label{eq:ca4}
\end{equation}
for all $B\in\mathscr{B}$. But since $f$ is locally of bounded variation,
(\ref{eq:ca4}) follows from the corresponding assumption for intervals,
i.e., $B=\left[x,y\right]$ for all $x<y$; so 
\begin{equation}
f\left(y\right)-f\left(x\right)=\int_{x}^{y}f^{\left(\mu\right)}d\mu.\label{eq:ca5}
\end{equation}
Condition (\ref{eq:ca5}) in turn is equivalent to the definition
of $f^{\left(\mu\right)}=\nabla_{\mu}f$ given in (\ref{eq:Dmax})
above.
\end{proof}
\begin{rem}
Let $f$ be a locally bounded variation function, and let $\mu$ be
a positive Borel measure. Suppose that the two measures $df$ and
$\mu$ are mutually singular; we then set $\nabla_{\mu}f=0$. See
eq (\ref{eq:ca10}) below for justification.
\end{rem}

\begin{rem}
We can decompose the restriction $df\ll\mu$ in the definition of
$\nabla_{\mu}f$ as follows: 
\begin{enumerate}
\item[(a)] Let $f$ and $\mu$ be as stated, and pass to the Jordan-decomposition
of the signed measure $df$ (as a Stieltjes measure). Then 
\begin{equation}
df=\left(\nabla_{\mu}f\right)d\mu+\left(df\right)_{s}\label{eq:ca10}
\end{equation}
where the term $\left(df\right)_{s}$ in (\ref{eq:ca10}) is mutually
singular w.r.t. $\mu$.
\item[(b)]  In section \ref{sec:class}, we shall consider a more detailed and
global analysis of (\ref{eq:ca10}) for a given Stieltjes measure
$df$. Indeed, when $df$ is given, then the second term on the RHS
in (\ref{eq:ca10}) will typically contain contributions from other
measures $\nu$, mutually singular, and each $\nu$ relatively singular
w.r.t. $\mu$.
\end{enumerate}
\end{rem}

\begin{lem}
\label{lem:T1}For all $f,g\in\mathscr{D}_{1}$, it holds that 
\begin{equation}
\nabla_{\mu}\left(fg\right)=f\nabla_{\mu}g+\left(\nabla_{\mu}f\right)g,\:\text{Leibnitz' rule}\label{eq:leibniz}
\end{equation}
and
\begin{equation}
\left(fg\right)\left(1\right)-\left(fg\right)\left(0\right)=\left\langle \nabla_{\mu}f,g\right\rangle _{L^{2}\left(\mu\right)}+\left\langle f,\nabla_{\mu}g\right\rangle _{L^{2}\left(\mu\right)}.\label{eq:gr1}
\end{equation}
\end{lem}

\begin{proof}
In our considerations below we make use of (\ref{eq:Dmax}), and the
definition (\ref{eq:C2}) for the new \textquotedblleft $\mu$-derivative.\textquotedblright{}
And we further make use of basic facts for the corresponding Stieltjes
measures; in particular the integration by parts formula for Stieltjes
measures. 

Details: If $f,g\in\mathscr{D}_{1}$ (see (\ref{eq:Dmax})), then
\begin{align*}
\int_{0}^{x}f\nabla_{\mu}g\,d\mu & =\int_{0}^{x}fdg\\
 & =fg\big|_{0}^{x}-\int_{0}^{x}g\,df\\
 & =fg\big|_{0}^{x}-\int_{0}^{x}g\nabla_{\mu}f\,d\mu.
\end{align*}
That is, 
\[
f\left(x\right)g\left(x\right)-f\left(0\right)g\left(0\right)=\int_{0}^{x}\left(f\nabla_{\mu}g+g\nabla_{\mu}f\right)d\mu,
\]
so that 
\[
\nabla_{\mu}\left(fg\right)=f\nabla_{\mu}g+\left(\nabla_{\mu}f\right)g,
\]
and (\ref{eq:gr1}) also follows.
\end{proof}
The proof above relies on key facts for Stieltjes integrals which
might perhaps not be widely known. For the benefit of readers, we
have therefore included the following alternative proof:
\begin{proof}[Second proof of \lemref{T1}]
Let $f,g\in\mathscr{D}_{1}$ as above, then 
\begin{align*}
f\left(1\right)g\left(1\right) & =\left(f\left(0\right)+\int_{0}^{1}\nabla_{\mu}fd\mu\right)\left(g\left(0\right)+\int_{0}^{1}\nabla_{\mu}gd\mu\right)\\
 & =f\left(0\right)g\left(0\right)+f\left(0\right)\int_{0}^{1}\nabla_{\mu}gd\mu+g\left(0\right)\int_{0}^{1}\nabla_{\mu}fd\mu\\
 & \quad+\left(\int_{0}^{1}\nabla_{\mu}fd\mu\right)\left(\int_{0}^{1}\nabla_{\mu}gd\mu\right),
\end{align*}
where
\begin{eqnarray*}
 &  & \left(\int_{0}^{1}\nabla_{\mu}fd\mu\right)\left(\int_{0}^{1}\nabla_{\mu}gd\mu\right)\\
 & = & \int_{0}^{1}\int_{0}^{1}\left(\nabla_{\mu}f\right)\left(s\right)\left(\nabla_{\mu}g\right)\left(t\right)\mu\left(ds\right)\mu\left(dt\right)\\
 & = & \int_{0}^{1}\left[\int_{0}^{t}\left(\nabla_{\mu}f\right)\left(s\right)\mu\left(ds\right)+\int_{t}^{1}\left(\nabla_{\mu}f\right)\left(s\right)\mu\left(ds\right)\right]\left(\nabla_{\mu}g\right)\left(t\right)\mu\left(dt\right)\\
 & = & \int_{0}^{1}\left[f\left(t\right)-f\left(0\right)\right]\left(\nabla_{\mu}g\right)\left(t\right)\mu\left(dt\right)+\int_{0}^{1}\left[\int_{t}^{1}\left(\nabla_{\mu}f\right)\left(s\right)\mu\left(ds\right)\right]\left(\nabla_{\mu}g\right)\left(t\right)\mu\left(dt\right)\\
 & = & \int_{0}^{1}\left[f\left(t\right)-f\left(0\right)\right]\left(\nabla_{\mu}g\right)\left(t\right)\mu\left(dt\right)+\int_{0}^{1}\left(\nabla_{\mu}f\right)\left(s\right)\left(g\left(s\right)-g\left(0\right)\right)\mu\left(ds\right).
\end{eqnarray*}
Thus, 
\[
\left(fg\right)\left(1\right)-\left(fg\right)\left(0\right)=\int_{0}^{1}f\nabla_{\mu}gd\mu+\int_{0}^{1}g\nabla_{\mu}fd\mu
\]
which is (\ref{eq:gr1}).
\end{proof}
\needspace{1em}
\begin{rem}
~
\begin{enumerate}
\item[(a)] In a $C^{*}$-algebraic framework, operators with dense domain and
satisfying a general Leibnitz rule of the form (\ref{eq:leibniz})
occur under the name \textquotedblleft unbounded derivations.\textquotedblright{}
They arise in a wider applied context, beyond that of fractal analysis,
and have been extensively studied. They play an important role in
dynamics, see e.g., \cite{MR887100}.
\item[(b)] A non-atomic measure $\mu$ is fixed, and we assume that $\mu$ is
supported in the unit interval $\left[0,1\right]$. We now turn to
the corresponding boundary-value problem for the operator $\nabla_{\mu}$,
see (\ref{eq:Tex}). Its operator theory will be identified relative
to the Hilbert space $L^{2}\left(\mu\right)$. To emphasize choice
of Hilbert space, we shall use the terminology $T_{\mu}$ for the
operator, and then add subscripts to indicate domains. The theory
of von Neumann (see \cite{MR1009163}) of selfadjoint extensions of
symmetric operators will be used. Only, for convenience, we shall
use the equivalent formulation in the form where we consider instead
skew-adjoint extensions of a fixed (minimal) skew-symmetric operator
with dense domain.
\end{enumerate}
\end{rem}

\begin{defn}
Set $T_{\mu,0}:L^{2}\left(\mu\right)\rightarrow L^{2}\left(\mu\right)$
by 
\begin{equation}
T_{\mu,0}=\nabla_{\mu}\big|_{\mathscr{D}_{0}}\label{eq:Tmin}
\end{equation}
where 
\begin{align}
\mathscr{D}_{0} & =C_{c}\left(J\right)\cap\mathscr{D}_{1}.\label{eq:Dmin}
\end{align}
\end{defn}

\begin{thm}
\label{thm:Tadj}The operator $T_{\mu,0}$ from (\ref{eq:Tmin})--(\ref{eq:Dmin})
is skew-symmetric, densely defined in the complex Hilbert space $L^{2}\left(\mu\right)$.
Moreover, $T_{\mu,0}$ has deficiency indices $\left(1,1\right)$,
and the corresponding skew-adjoint extensions are specified by 
\begin{equation}
dom\left(T_{\mu,\alpha}\right)=\left\{ f\in\mathscr{D}_{1}:f\left(1\right)=\alpha f\left(0\right)\right\} ,\;\left|\alpha\right|=1,\label{eq:Dex}
\end{equation}
and 
\begin{equation}
T_{\mu,\alpha}=\nabla_{\mu}\big|_{dom\left(T_{\mu,\alpha}\right)}.\label{eq:Tex}
\end{equation}
\end{thm}

\begin{proof}
From the identity 
\[
\left\langle \nabla_{\mu}f,g\right\rangle _{L^{2}\left(\mu\right)}+\left\langle f,\nabla_{\mu}g\right\rangle _{L^{2}\left(\mu\right)}=\left(\overline{f}g\right)\left(1\right)-\left(\overline{f}g\right)\left(0\right),
\]
which is valid for all $f,g\in\mathscr{D}_{1}$, it follows that the
right-hand side vanishes if and only if $f,g$ are in $dom\left(T_{\mu,\alpha}\right)$,
see (\ref{eq:Dex}). 
\end{proof}
For more details on \emph{extensions of skew-symmetric operators},
we refer to \cite{MR1009163,MR2953553}.

We now turn to the \emph{unitary one-parameter groups} which are generated
by the skew-adjoint extension operators above, from \thmref{Tadj},
eq. (\ref{eq:Tex}).

Consider the two unitary one parameter groups with the respective
skew adjoint generators, and periodic boundary condition $f\left(0\right)=f\left(1\right)$.
\begin{equation}
L^{2}\left(\left[0,1\right],\lambda\right)\ni\psi\xrightarrow{\;U_{\lambda}\left(t\right)\;}\psi\left(\left[\cdot+t\right]_{F}\right)\in L^{2}\left(\left[0,1\right],\lambda\right)
\end{equation}
where $\left[\cdot\right]_{F}$ denotes the fractional part of a real
number. 

Let $\mu$ be a non-atomic Borel measure on $\left[0,1\right]$, and
let 
\begin{equation}
g\left(x\right)=\mu\left(\left[0,x\right]\right).\label{eq:tm1-1}
\end{equation}

\begin{figure}[H]
\includegraphics[width=0.4\paperwidth]{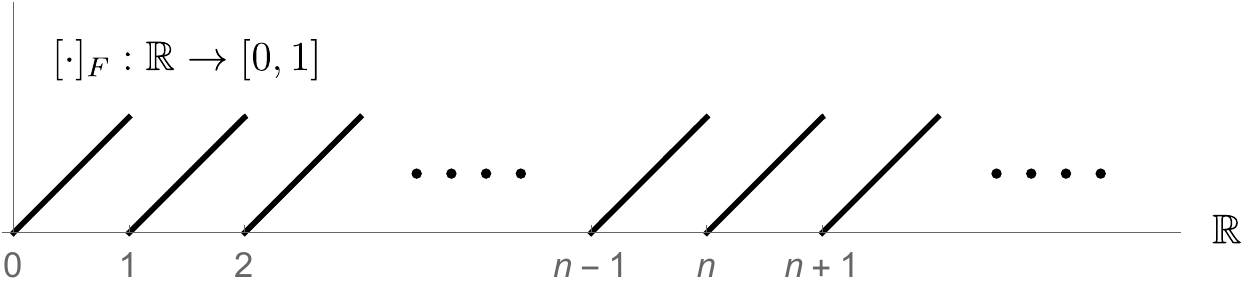}

\caption{$\left[\cdot\right]_{F}:\mathbb{R}\rightarrow\left[0,1\right]$}
\end{figure}

\begin{equation}
\xymatrix{ &  & \psi\left(\left[\cdot+t\right]_{F}\right)\ar[rr]^{W_{\mu}} &  & \psi\left(\left[g\left(\cdot\right)+t\right]_{F}\right)\\
\psi\ar@/^{1.3pc}/[rru]^{U_{\lambda}\left(t\right)}\ar@/_{1.3pc}/[rrd]_{W_{\mu}}\\
 &  & \psi\circ g\ar@/_{1.3pc}/[rruu]_{U_{\mu}\left(t\right)}
}
\end{equation}
We have 
\begin{eqnarray}
U_{\mu}\left(t\right)\left(\psi\circ g\right)\left(\cdot\right) & = & \psi\left(\left[g\left(\cdot\right)+t\right]_{F}\right)\nonumber \\
 & \Updownarrow\\
U_{\mu}\left(t\right)W_{\mu} & = & W_{\mu}U_{\lambda}\left(t\right),\nonumber 
\end{eqnarray}
and summarized in the diagram below:
\[
\xymatrix{U_{\lambda}\left(t\right)\psi\in L^{2}\left(\lambda\right)\ar[rr]^{W_{\mu}} &  & U_{\mu}\left(t\right)f\in L^{2}\left(\mu\right)\\
\psi\in L^{2}\left(\lambda\right)\ar[u]^{U_{\lambda}\left(t\right)}\ar[rr]_{W_{\mu}} &  & f\in L^{2}\left(\mu\right)\ar[u]_{U_{\mu}\left(t\right)}
}
\]

\begin{lem}
\label{lem:tc}Fix $\mu$ and $g$, and set  $W_{\mu}\psi=\psi\circ g$.
Then TFAE:
\begin{enumerate}
\item $W_{\mu}U_{\lambda}\left(t\right)=U_{\mu}\left(t\right)W_{\mu}:L^{2}\left(\lambda\right)\rightarrow L^{2}\left(\mu\right)$
\item $U_{\mu}\left(t\right)=W_{\mu}U_{\lambda}\left(t\right)W_{\mu}^{*}:L^{2}\left(\mu\right)\rightarrow L^{2}\left(\mu\right)$
\item $U_{\lambda}\left(t\right)=W_{\mu}^{*}U_{\mu}\left(t\right)W_{\mu}:L^{2}\left(\lambda\right)\rightarrow L^{2}\left(\lambda\right)$
\end{enumerate}
\end{lem}

Starting with $\mu$, specified as before, we then set $g=g_{\mu}$,
$g\left(x\right):=\mu\left(\left[0,x\right]\right)$, the ``\emph{cumulative
distribution}''. It follows that then the operator $W_{\mu}$, given
by $W_{\mu}\psi:=\psi\circ g$, will be an isometric isomorphism of
$L^{2}\left(\lambda\right)$ onto $L^{2}\left(\mu\right)$, with adjoint
$W_{\mu}^{*}:L^{2}\left(\mu\right)\rightarrow L^{2}\left(\lambda\right)$,
given by (\ref{eq:tc3}). We add that a detailed analysis of this
operator $W_{\mu}$, and its applications, will be undertaken below. 
\begin{cor}[Time-change]
 We have 
\begin{align}
\left(U_{\mu}\left(t\right)f\right)\left(x\right) & =W_{\mu}U_{\lambda}\left(t\right)W_{\mu}^{*}f\left(x\right)\\
 & =W_{\mu}^{*}f\left(\left[g\left(x\right)+t\right]_{F}\right),\label{eq:tc2}
\end{align}
where
\begin{equation}
W_{\mu}\psi\left(\cdot\right)=\psi\circ g,
\end{equation}
and
\begin{align}
\left(W_{\mu}^{*}f\right)\left(y\right) & =\int_{g^{-1}\left(\left\{ y\right\} \right)}f\,d\rho_{y},\label{eq:tc3}
\end{align}
and so
\begin{equation}
W_{\mu}^{*}f\left(\left[g\left(x\right)+t\right]_{F}\right)=\int_{g^{-1}\left(\left\{ \left[g\left(x\right)+t\right]_{F}\right\} \right)}fd\rho_{\left[g\left(x\right)+t\right]_{F}}.
\end{equation}
\end{cor}

\begin{figure}
\includegraphics[width=0.35\columnwidth]{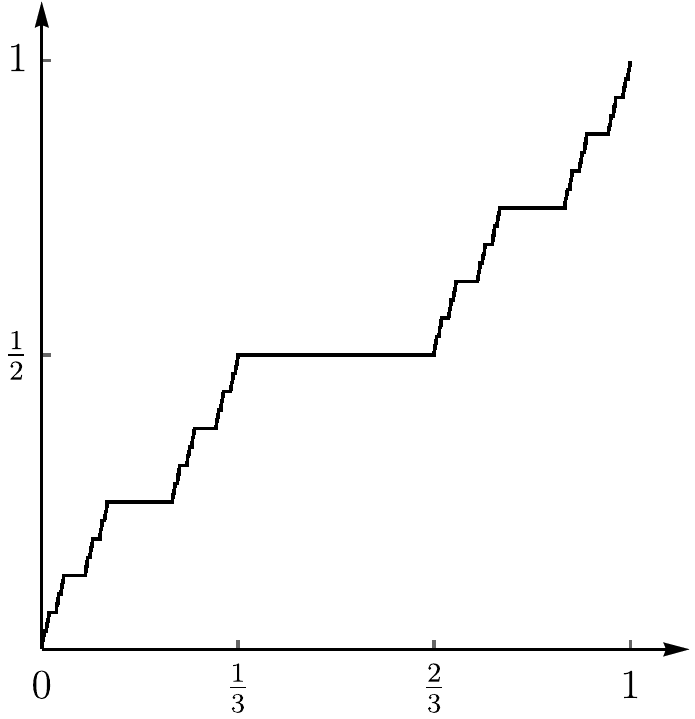}

\caption{\label{fig:devil}$g_{\mu_{3}}\left(\cdot\right)$ the \emph{middle
third Cantor measure} as a Stieltjes measure. See \remref{sm}.}
\end{figure}

In (\ref{eq:tc3}), $d\rho_{y}$ denote conditional measures. Since
$\mu\circ g^{-1}=\lambda$ we have conditional measures $\left\{ \rho_{y}\right\} _{y\in J}$
subject to the partition $\left\{ g^{-1}\left(\left\{ y\right\} \right)\right\} _{y\in J}$.
By the disintegration theorem, applied to $\mu$, we therefore get
the representation:
\begin{equation}
\mu\left(\cdot\right)=\int_{J}\rho_{y}\left(\cdot\right)dy.\label{eq:DT1}
\end{equation}
For details regarding (\ref{eq:DT1}), we refer readers to the literature,
e.g., \cite[ch2.  pg 13-21]{MR3793614} and \cite{MR0030584,MR0047744,MR0228654,MR0168697}. 
\begin{rem}[The middle-third Cantor measure]
\label{rem:sm}~
\begin{enumerate}
\item[(a)] For a detailed account of harmonic analyses on fractals, see \cite{MR3838440,MR3996038};
as well as the papers cited there. 

Below we discuss the special case when $\mu=\mu_{3}$ is assumed to
be the middle-third Cantor measure. In details, if the pair $\sigma_{0}$,
$\sigma_{1}$ denotes $\sigma_{0}\left(x\right)=\frac{x}{3}$, $\sigma_{1}\left(x\right)=\frac{x+2}{3}$,
$x\in\mathbb{R}$; then $\mu_{3}$ is the corresponding normalized
IFS-measure fixed by 
\begin{equation}
\mu_{3}=\frac{1}{2}\left(\mu_{3}\circ\sigma_{0}^{-1}+\mu_{3}\circ\sigma_{1}^{-1}\right),\label{eq:CAN1}
\end{equation}
supported in the Cantor set. The cumulative distribution function
of $\mu_{3}$ is the function $g\left(x\right):=\mu_{3}\left(\left[0,x\right]\right)$.
It is sketched in \figref{devil}.
\item[(b)]  We show in section \ref{subsec:aifs} (see especially \thmref{D4},
eq. (\ref{eq:Dd2}), and \figref{subd}) that the Cantor-measure $\mu$
solving equation (\ref{eq:CAN1}) arises as a special case of a wider
class of self-similar measures, also called IFS-measures. In the general
case of IFS-measures, the corresponding equation is (\ref{eq:Dd2}).
It is defined from a finite system of endomorphisms $\sigma_{i}$
(called function systems), and the corresponding iterated function
system (IFS)-measure $\mu$ will then arise as a Markov chain-average
(\figref{subd}) of its corresponding $\sigma_{i}$ transforms. In
this general case, the solution $\mu$ can be found, see \thmref{D4}:
Every IFS-measure $\mu$ allows a representation (\ref{eq:Dd8}),
defined as a pull-back of an infinite-product measure (\ref{eq:Dd5}).
\item[(c)]  Note that the function $g$ defined this way (see (\ref{eq:tm1-1}))
will have the following properties: It is monotone (increasing, or
more precisely non-decreasing). Moreover, by standard measure theory,
it follows that the initial measure $\mu$ will then agree with the
corresponding Stieltjes measure (here denoted $dg$), i.e., we have
$\mu=dg$. Hence it is possible to define branches of an inverse to
the function $g$, i.e., $g^{-1}$ defined a.e., w.r.t. $\mu$. Intuitively
we think of the function $g$ as a time-change, see (\ref{eq:tc2}). 
\item[(d)]  In the case when $\mu$ is the standard middle third Cantor measure,
then the corresponding function $g$ is illustrated in \figref{devil}
(often called \textquotedblleft the Devil\textquoteright s staircase.\textquotedblright )
At each iteration in the construction of $g$, we take it to be constant
on the omitted middle-third intervals.

In more details, let $\sigma_{0},\sigma_{1}:\left[0,1\right]\rightarrow\left[0,1\right]$
be given by 
\begin{equation}
\sigma_{0}\left(x\right)=\frac{x}{3},\quad\sigma_{1}\left(x\right)=\frac{x+2}{3}.\label{eq:CAN2}
\end{equation}
Then, for the Cantor set $C_{3}$, we have 
\[
\sigma_{0}\left(C_{3}\right)\dot{\cup}\sigma_{1}\left(C_{3}\right)=C_{3}.
\]
Using bit representations, we get 
\[
g^{-1}\left(\sum_{n=1}^{\infty}\frac{b_{n}}{2^{n}}\right)=\sum_{n=1}^{\infty}\frac{2b_{n}}{3^{n}},\quad b_{n}\in\left\{ 0,1\right\} .
\]
To see this, recall that every $x\in C_{3}$ has the following representation:
\begin{align*}
x & =\sum_{n=1}^{\infty}\frac{a_{n}}{3^{n}}\qquad\left(a_{n}\in\left\{ 0,2\right\} \right)\\
 & =\frac{a_{1}}{3}+\sum_{n=2}^{\infty}\frac{a_{n}}{3^{n}}\\
 & =\frac{a_{1}}{3}+\frac{1}{3}\underset{\text{shift of bits}}{\underbrace{\left(\frac{a_{2}}{3}+\frac{a_{3}}{3^{2}}+\cdots\right)}}.
\end{align*}

\item[(e)]  Let $\mu$ $(=\mu_{3})$ be the Cantor measure of (\ref{eq:CAN1})
above. We then get the following transformation of pairs of Borel
measures on the unit-interval $J$: $\mu\circ g^{-1}=\lambda_{1}$,
where $\lambda_{1}$ is Lebesgue measure on $J$.

Note, the defining properties of the Cantor measure $\mu$, $supp\left(\mu\right)=C_{3}$,
and the Lebesgue measure $\lambda$ are as follows: 
\begin{alignat*}{2}
\text{Cantor} &  & \quad & \int\varphi\,d\mu=\frac{1}{2}\left(\int\varphi\circ\sigma_{0}\,d\mu+\int\varphi\circ\sigma_{1}\,d\mu\right),\;\forall\varphi\in C;\:\text{see \ensuremath{\left(\ref{eq:CAN1}\right)}-\ensuremath{\left(\ref{eq:CAN2}\right)} above;}\\
\text{Lebesgue} &  &  & \int\varphi\,d\lambda=\frac{1}{2}\left(\int\varphi\left(\frac{x}{2}\right)d\lambda\left(x\right)+\int\varphi\left(\frac{x+1}{2}\right)d\lambda\left(x\right)\right),\;\forall\varphi\in C.
\end{alignat*}

\end{enumerate}
\end{rem}

\subsection{A symmetric pair of operators for $L^{2}(\mu)$ and $L^{2}(\nu)$}

Let $\mu$, $\nu$ be two non-atomic measures on $J=\left[0,1\right]$. 
\begin{lem}
\label{lem:gp}For $f\in L^{2}\left(\nu\right)\cap dom\left(T_{\mu}\right)$
and $g\in L^{2}\left(\mu\right)\cap dom\left(T_{\nu}\right)$, we
have
\[
\left(fg\right)\left(x\right)-\left(fg\right)\left(0\right)=\int_{0}^{x}\left(\nabla_{\mu}f\right)gd\mu+\int_{0}^{x}f\left(\nabla_{\nu}g\right)d\nu.
\]
\end{lem}

\begin{proof}
A direct computation: 
\begin{eqnarray*}
\int_{0}^{x}\left(\nabla_{\mu}f\right)gd\mu & = & \int_{0}^{x}gdf\\
 & = & gf\big|_{0}^{x}-\int_{0}^{x}fdg\\
 & = & gf\big|_{0}^{x}-\int_{0}^{x}f\left(\nabla_{\nu}g\right)d\nu\\
 & \Downarrow\\
\left(fg\right)\left(x\right)-\left(fg\right)\left(0\right) & = & \int_{0}^{x}\left(\nabla_{\mu}f\right)gd\mu+\int_{0}^{x}f\left(\nabla_{\nu}g\right)d\nu\\
 & \Updownarrow\\
d\left(fg\right) & = & \left(\nabla_{\mu}f\right)gd\mu+f\left(\nabla_{\nu}g\right)d\nu.
\end{eqnarray*}
See also the proof of \lemref{T1}, and that of \thmref{D2}.
\end{proof}
\begin{cor}
Given a pair of non-atomic measures $\mu$ and $\nu$ as described;
with \lemref{gp} and the arguments in sect \ref{sec:DP}, and \ref{subsec:tl},
we then arrive at the following dual-pair realization for the associated
operators:
\begin{equation}
\xymatrix{\overline{dom\left(T_{\mu}\right)\cap L^{2}\left(\nu\right)}^{L^{2}\left(\nu\right)}\ar@/^{1.7pc}/[rr]^{T_{\mu}} &  & \overline{dom\left(T_{\nu}\right)\cap L^{2}\left(\mu\right)}^{L^{2}\left(\mu\right)}\ar@/^{1.5pc}/[ll]^{-T_{\nu}}}
\label{eq:munu}
\end{equation}
\begin{equation}
T_{\mu}\subset-T_{\nu}^{*},\quad-T_{\nu}\subset T_{\mu}^{*}.
\end{equation}
\end{cor}

\begin{rem}
Let $K_{F}=\frac{d}{d\mu}\frac{d}{dx}$ be as before. Note it has
a quadratic form representation as follows: 
\begin{equation}
\left\langle \varphi,K_{F}\psi\right\rangle _{L^{2}\left(\mu\right)}=-\left\langle \varphi',\psi'\right\rangle _{L^{2}\left(\lambda\right)}
\end{equation}
where $\varphi'=\frac{d\varphi}{dx}$, $\psi'=\frac{d\psi}{dx}$.
This follows from the dual pair $d/dx$, $d/d\mu$ in \thmref{D2}.
Details: 
\begin{equation}
\int\varphi\left(\frac{d}{d\mu}\frac{d}{dx}\psi\right)d\mu=-\int\frac{d}{dx}\varphi\frac{d}{dx}\psi d\lambda=-\int\varphi'\psi'd\lambda.
\end{equation}

Then we get the selfadjoint operator 
\begin{equation}
T_{\mu}T_{\mu}^{*}=V\left(T_{\mu}^{*}T_{\mu}\right)V^{*}
\end{equation}
where $V$ is a partial isometry, and $T_{\mu}T_{\mu}^{*}$ has dense
domain in $L^{2}\left(\mu\right)$. (See \thmref{vN} and \defref{sp})
So $T_{\mu}T_{\mu}^{*}$ is a selfadjoint operator extension for the
quadratic form $QF\left(K_{F}\right)$, where 
\begin{equation}
QF\left(K_{F}\right)\left(\varphi,\psi\right)=-\int\varphi'\psi'd\lambda.
\end{equation}

Similarly, in (\ref{eq:munu}) we get 
\begin{equation}
QF\left(K_{F}\right)\subset T_{\mu}T_{\mu}^{*}.
\end{equation}
\end{rem}

The notion of \textquotedblleft extension\textquotedblright{} of a
closable positive quadratic form $Q$ is made precise in, for example
\cite{MR1335452}. It is a quadratic form-version of the analogous
extension of Friedrichs (see e.g., \cite{MR1009163}). If $K$ is
such a selfadjoint extension operator for $Q$, then the requirement
is that the domain of $Q$ be contained in the domain of the square
root of $K$.
\begin{example}[see also \figref{devil},  (\ref{eq:CAN2}), and \exaref{DE1} (\ref{enu:de2})]
Let $\mu=\mu_{3}$, the middle $1/3$ Cantor measure, 
\begin{equation}
\mu_{3}=\frac{1}{2}\left(\mu_{3}\circ\sigma_{0}^{-1}+\mu_{3}\circ\sigma_{1}^{-1}\right),
\end{equation}
supported on the Cantor set 
\begin{equation}
C_{3}=\left[0,1\right]\backslash\bigcup\left\{ \text{middle intervals}\right\} ,
\end{equation}
so that $\lambda\left(C_{3}\right)=0$. Let $g_{3}\left(x\right)=\mu_{3}\left(\left[0,x\right]\right)$.
Set $K_{F}^{\left(3\right)}=\frac{d}{d\mu_{3}}\frac{d}{dx}$. Then
\begin{equation}
K_{F}^{\left(3\right)}\left(\psi\circ g_{3}\right)=\frac{d}{d\mu}\left(\psi'\circ g\right)\left(\frac{d}{dx}g_{3}\right)=0
\end{equation}
since $\frac{d}{dx}g_{3}=0$ in the sense of distribution, i.e, $g'_{3}=0$
a.e. $\lambda$. 
\end{example}

\subsection{The $L^{2}\left(\mu\right)$-boundary value problem}

We now turn to a detailed harmonic analysis of the skew-adjoint extension
operators introduced in \thmref{Tadj}. Recall that, when $\alpha$
is fixed (on the complex circle), then the corresponding skew-adjoint
extension operator generates a \emph{unitary one-parameter group}
(depending on $\alpha$) of operators $U(t)$ acting in $L^{2}(\mu)$.
The harmonic analysis of this unitary one-parameter group was presented
in detail above in \lemref{tc}. The following result offers a complete
spectral picture.
\begin{lem}
\label{lem:T9}Set 
\begin{equation}
v_{x}\left(y\right):=\mu\left(\left[0,y\wedge x\right]\right).\label{eq:d1}
\end{equation}
Then we have 
\begin{equation}
T_{\mu}v_{x}=\frac{dv_{x}}{d\mu}=\chi_{\left[0,x\right]}.\label{eq:d2}
\end{equation}

Moreover, for any $F\in C^{1}$, we get
\begin{equation}
T_{\mu}\left(F\left(v_{x}\right)\right)=F'\left(v_{x}\right)\chi_{\left[0,x\right]}.\label{eq:d3}
\end{equation}
In particular, 
\begin{equation}
T_{\mu}\left(e^{iv_{x}}\right)=ie^{iv_{x}}\chi_{\left[0,x\right]}.\label{eq:d4}
\end{equation}
\end{lem}

\begin{proof}
Note that $v_{x}\left(y\right)=\int_{0}^{y}\chi_{\left[0,x\right]}\left(s\right)d\mu\left(s\right)=\mu\left(\left[0,y\wedge x\right]\right)$,
which is (\ref{eq:d2}). 

Now, if $F\in C^{1}$ then 
\[
d\left(F\left(v_{x}\right)\right)=F'\left(v_{x}\right)dv_{x}.
\]
That is, 
\begin{align*}
F\left(v_{x}\left(y\right)\right)-F\left(v_{x}\left(0\right)\right) & =\int_{0}^{y}F'\left(v_{x}\right)dv_{x}\\
 & =\int_{0}^{y}F'\left(v_{x}\right)\chi_{\left[0,x\right]}d\mu,
\end{align*}
so that (\ref{eq:d3}) holds, and (\ref{eq:d4}) follows from this. 

We now turn to the detailed \emph{spectral expansion} for the indexed
system of skew-adjoint operators $T_{\mu,\theta}$ discussed in \thmref{Tadj}.
\end{proof}
\begin{thm}
Let $T_{\mu,\theta}$ ($\alpha=e^{i\theta}$) be as above. In particular,
elements in $dom\left(T_{\mu,\theta}\right)$ satisfy the boundary
condition 
\begin{equation}
f\left(1\right)=e^{i\theta}f\left(0\right),\quad\theta\in\mathbb{R}.\label{eq:tb1-1}
\end{equation}
Then $T_{\mu,\theta}$ has the following spectral representation:
\begin{equation}
-iT_{\mu,\theta}=\sum_{n\in\mathbb{Z}}\lambda_{n}\left|\varphi_{n}\left\rangle \right\langle \varphi_{n}\right|
\end{equation}
where 
\begin{equation}
\lambda_{n}=\frac{\theta+2n\pi}{\mu\left(J\right)},\quad\varphi_{n}\left(x\right)=\frac{1}{\sqrt{\mu\left(J\right)}}e^{i\lambda_{n}\mu\left(\left[0,x\right]\right)},\quad n\in\mathbb{Z},\label{eq:C26}
\end{equation}
and $\left\{ \varphi_{n}\right\} $ is an ONB in $L^{2}\left(\mu\right)$.
And the associated unitary one-parameter group $U\left(t\right)=e^{tT_{\mu,\theta}}$
is given by 
\begin{equation}
e^{tT_{\mu,\theta}}=\sum_{n\in\mathbb{Z}}e^{it\lambda_{n}}\left|\varphi_{n}\left\rangle \right\langle \varphi_{n}\right|.
\end{equation}
 
\end{thm}

\begin{proof}
With \lemref{T9} we justify the eigenvalue/eigenfunction assertions
in (\ref{eq:C26}) in the Theorem. Note that 
\begin{equation}
T_{\mu,\theta}f=i\lambda f\Longleftrightarrow f\left(x\right)-f\left(0\right)=i\lambda\int_{0}^{x}fd\mu.\label{eq:tem1}
\end{equation}
It suffices to verify that  $f\left(x\right)=e^{i\lambda\mu\left(\left[0,x\right]\right)}$
satisfies (\ref{eq:tem1}). Indeed, 
\[
\int_{0}^{x}e^{i\lambda\mu\left(\left[0,s\right]\right)}d\mu\left(s\right)=\int_{0}^{x}e^{i\lambda g\left(s\right)}dg\left(s\right)=\frac{1}{i\lambda}\left(e^{i\lambda g\left(x\right)}-1\right),
\]
where $g\left(x\right)=\mu\left(\left[0,x\right]\right)$ as before. 
\end{proof}
\begin{lem}
The adjoint operator of $T_{\mu,0}$ is given by $T_{\mu,0}^{*}=-\nabla_{\mu}$,
defined on $\mathscr{D}_{1}$. 
\end{lem}

\begin{proof}
For any $f\in\mathscr{D}_{0}$ and $g\in\mathscr{D}_{1}$, one has
\[
\int_{0}^{1}\left(T_{\mu,0}f\right)gd\mu+\int_{0}^{1}f\nabla_{\mu}gd\mu=0
\]
and so $T_{\mu,0}^{*}\subset-\nabla_{\mu}\big|_{\mathscr{D}_{1}}$. 

Conversely, for all $g\in dom\left(T_{\mu,0}^{*}\right)\subset L^{2}\left(\mu\right)$,
let $h=T_{\mu,0}^{*}g$ then 
\[
\int_{0}^{1}\left(T_{\mu,0}f\right)gd\mu=\int_{0}^{1}fhd\mu,\quad\forall f\in\mathscr{D}_{0}.
\]
Set $H\left(x\right)=H\left(0\right)+\int_{0}^{x}hd\mu\in\mathscr{D}_{1}$,
then 
\[
\int_{0}^{1}\left(T_{\mu,0}f\right)gd\mu=-\int_{0}^{1}f\nabla_{\mu}Hd\mu=-\int_{0}^{1}\left(T_{\mu,0}f\right)Hd\mu,\quad\forall f\in\mathscr{D}_{0}.
\]
It follows that $g=-H$ in $L^{2}\left(\mu\right)$. 
\end{proof}
Even though our present focus is on the case when $\mu$ is assumed
singular w.r.t. Lebesgue measure, $\lambda$, for the sake of illustration,
the next two results, \lemref{abs} and \corref{C13}, cover the other
extreme, i.e., when $\mu\ll\lambda$ holds. 
\begin{lem}
\label{lem:abs}Assume $\mu\ll dx$, where $dx=$ the Lebesgue measure
on $\left[0,1\right]$, and let $d\mu/dx=M\left(x\right)$. Then 
\[
\nabla_{\mu}f\left(x\right)=\left(M^{-1}f'\right)\left(x\right).
\]
\end{lem}

\begin{proof}
Indeed, if $f\in\mathscr{D}_{1}$ then 
\begin{align*}
f\left(x\right)-f\left(0\right) & =\int_{0}^{x}f'\left(x\right)dx\\
 & =\int_{0}^{x}\left(f'M^{-1}\right)\left(x\right)M\left(x\right)dx=\int_{0}^{x}\nabla_{\mu}f\left(x\right)d\mu\left(x\right).
\end{align*}
\end{proof}
\begin{cor}
\label{cor:C13}Let $M\left(x\right)=d\mu/dx$ be as above. The deficiency
subspaces of $T_{\mu,0}$ are determined as follows:
\begin{align*}
\mathscr{D}_{\pm}\left(T_{\mu,0}\right) & :=\left\{ g:\left(g'M^{-1}\right)\left(x\right)=\pm g\right\} \\
 & =span\left\{ \exp\left(\pm\int M\left(x\right)dx\right)\right\} .
\end{align*}
In particular, this implies that $T_{\mu,0}$ has deficiency indices
$\left(1,1\right)$. 
\end{cor}

\begin{proof}
One checks that 
\begin{align*}
\nabla_{\mu}g & =\pm g\\
 & \Updownarrow\\
g'M^{-1} & =\pm g\\
 & \Updownarrow\\
\left(\ln g\right)' & =\pm M
\end{align*}
and the conclusion follows.
\end{proof}
\begin{defn}
Set 
\[
\Delta_{\mu}=\frac{d}{d\mu}\frac{d}{dx}=\nabla_{\mu}\frac{d}{dx}
\]
defined on 
\begin{align*}
dom\left(\Delta_{\mu}\right) & :=\left\{ f\mid f'\in\mathscr{D}_{1}\right\} \\
 & =\left\{ c+\int_{0}^{t}\left(f\left(0\right)+\int_{0}^{x}\nabla_{\mu}fd\mu\right)dx\mid\nabla_{\mu}f\in L^{2}\left(\mu\right)\right\} .
\end{align*}
 Set 
\[
\Delta_{\mu,0}=\Delta_{\mu}\big|_{C_{c}\cap dom\left(\Delta_{\mu}\right)}.
\]
\end{defn}

Then $\Delta_{\mu,0}$ has two particular selfadjoint extensions that
correspond to Dirichlet and Neumann boundary conditions. 
\begin{thm}
\label{thm:C1}Fix $\mu$. Let $K_{F}$ denote a corresponding selfadjoint
realization of $\Delta_{\mu}$. For every $g\in L^{2}\left(\mu\right)$,
set 
\begin{equation}
f\left(t\right):=\int_{0}^{t}\left(\int_{0}^{x}gd\mu\right)dx.
\end{equation}
Then 
\begin{equation}
\left(K_{F}f\right)\left(t\right)=\left(f'\right)^{\mu}=g,
\end{equation}
and the eigenvalue problem may be stated as 
\begin{equation}
g'\left(t\right)=\lambda\int_{0}^{t}gd\mu.
\end{equation}
\end{thm}

\begin{proof}
We have
\begin{eqnarray*}
\left(K_{F}f\right)\left(t\right) & = & \lambda f\\
 & \Updownarrow\\
g & = & \lambda f\\
 & \Updownarrow\\
g\left(t\right) & = & \lambda\int_{0}^{t}\left(\int_{0}^{x}gd\mu\right)dx.
\end{eqnarray*}
And so $g'\left(t\right)=\lambda\int_{0}^{t}gd\mu$.
\end{proof}
\begin{cor}
Assume $\mu\ll dx$, and $d\mu/dx=M>0$. Let 
\[
dW_{t}^{\left(\mu\right)}=M^{-1/2}\left(x\right)dB_{t},
\]
where $B_{t}$ is standard Brownian motion. Then, 
\[
u\left(t,x\right):=\mathbb{E}_{x}\left(f\left(W_{t}^{\left(\mu\right)}\right)\right)
\]
satisfies 
\[
\frac{\partial u}{\partial t}=\frac{1}{2}K_{F}u.
\]
\end{cor}

\begin{proof}
An application of Ito's lemma (see e.g., \cite{MR562914}) gives 
\[
df\left(W_{t}\right)=f'\left(W_{t}\right)M^{-1/2}dB_{t}+\frac{1}{2}f''\left(W_{t}\right)M^{-1}dt
\]
and 
\[
\frac{d}{dt}\mathbb{E}\left(f\left(W_{t}\right)\right)\big|_{t=0}=\frac{1}{2}M^{-1}\left(x\right)f''\left(x\right)=\frac{1}{2}K_{F}f\left(x\right),
\]
where $K_{F}f=M^{-1}f''$, by \lemref{abs}. 
\end{proof}

\subsection{\label{subsec:KFD}Krein-Feller diffusion}

We now turn to a detailed discussion of diffusion, and its connection
to the Krein-Feller operator. As before, our starting point is a fixed
measure $\mu$ supported on the real line, and we let $K_{F}$ be
the corresponding Krein-Feller operator. The measure $\mu$ under
consideration is assumed to be defined on the Borel-sigma algebra,
and further assumed positive, sigma-finite, and non-atomic. The measure
$\mu$ may be singular, and the main novelty in our analysis of the
diffusion semigroup is for the singular case; including the case of
IFS-measures. We first introduce the \emph{centered Gaussian process}
$W^{\left(\mu\right)}$ having $\mu$ as its quadratic variation.
We then note that Ito's lemma applies to $W^{\left(\mu\right)}$,
see (\ref{eq:C3-6}). We further study the Markov semigroup ((\ref{eq:CC-8})
and \lemref{Scont}) corresponding to $W^{\left(\mu\right)}$. In
our two main results \lemref{C9} and \thmref{Cc16} below, we identify
the selfadjoint extension of $K_{F}$ from section \ref{sec:DP} as
the \emph{infinitesimal generator} for this diffusion semigroup, \lemref{C9}.
In \thmref{Cc16} we introduce a \emph{time-change} in our characterization
of the \emph{diffusion semigroup}.

Below we show that every positive non-atomic Borel measure (see \lemref{C1})
gives rise to a naturally associated dual pair of operators, as per
\defref{sp}. The dual pair is made precise in (\ref{eq:cc29}), and
the figure in \thmref{D2}.

Fix a measure space $\left(J,\mathscr{B},\mu\right)$ with $J=[0,\infty)$.
We introduce the corresponding positive definite kernel: 
\begin{equation}
K_{\mu}\left(x,y\right)=\mu\left(\left[0,x\wedge y\right]\right).\label{eq:C2-1}
\end{equation}
Note that if $\mu=\lambda=dx=$ the Lebesgue measure then 
\begin{equation}
K_{\lambda}\left(x,y\right)=x\wedge y,\label{eq:C2-2}
\end{equation}
the usual \emph{covariance function }for Brownian motion.

Starting with $\mu$ (assumed non-atomic) and the RKHS $\mathscr{H}\left(\mu\right)$,
we arrive at a generalized Brownian motion $W_{x}^{\left(\mu\right)}$;
i.e., a Gaussian process with 
\begin{equation}
\mathbb{E}\left(W_{x}^{\left(\mu\right)}\right)=0,
\end{equation}
and 
\begin{equation}
\mathbb{E}\left(W_{x}^{\left(\mu\right)}W_{y}^{\left(\mu\right)}\right)=K_{\mu}\left(x,y\right),\quad\forall x,y\in[0,\infty).\label{eq:C2-4}
\end{equation}

A detailed description of $\{W_{x}^{\left(\mu\right)}\}$ and its
Ito-calculus is contained in many relevant papers, and books; see
e.g., \cite{MR4106884,MR4020693,MR3687240,MR3402823,MR2966130,MR2793121,MR0345224,MR154338,MR2571742,doi:10.1142/11980}.

Here we shall need the following: Let $\varphi\in C^{2}$, and consider
the corresponding diffusion semigroup (depending on $\mu$): 
\begin{equation}
u\left(t,x\right)\coloneqq\mathbb{E}\left(\varphi\left(W_{t}^{\left(\mu\right)}\right)\mid W_{0}^{\left(\mu\right)}=x\right)\label{eq:C2-5}
\end{equation}
where $\mathbb{E}$ in (\ref{eq:C2-5}) refers to the expectation
$\mathbb{E}\left(\cdot\right)=\int_{\Omega}\left(\cdot\right)d\mathbb{P}$
for the probability space associated to (\ref{eq:C2-1}). Then $\mathbb{E}(\cdots\mid W_{0}^{\left(\mu\right)}=x)$
in (\ref{eq:C2-5}) refers to conditioning with all paths $\omega$
s.t. $\omega\left(0\right)=x$. Here we also use the representation
\begin{equation}
W_{t}^{\left(\mu\right)}\left(\omega\right)=\omega\left(t\right),\quad\forall\omega\in\Omega\label{eq:C2-6}
\end{equation}
for the Gaussian process $\{W_{t}^{\left(\mu\right)}\}$. Since the
Gaussian process $W^{\left(\mu\right)}$ has independent increments,
it follows that eq. (\ref{eq:C2-5}) defines a semigroup (see e.g.,
\cite{MR0345224,MR126625,MR2973393,MR3571410}), and we shall call
it the \emph{Markov semigroup}. Its properties and its infinitesimal
generator will be identified in \lemref{C9} below, and in the subsequent
discussion.

The Gaussian process from (\ref{eq:C2-4}) and (\ref{eq:C2-6}) is
often called a generalized Brownian motion, or a Gaussian field. The
associated Ito-integral is also used in the proof of \lemref{C9}
below.

We now introduce the Krein-Feller operator 
\begin{equation}
K_{F}\coloneqq\frac{\partial^{2}}{\partial\mu\partial x}\label{eq:C2-7}
\end{equation}
(see above.) In our discussion of (\ref{eq:C2-5}), we shall consider
$K_{F}$ as acting in the $x$-variable.
\begin{lem}
\label{lem:C9}The diffusion (\ref{eq:C2-5}) is generated by the
following generalized heat equation: 
\begin{equation}
\frac{\partial u}{\partial t}=\frac{1}{2}K_{F}u,\quad u\left(0,x\right)=\varphi\left(x\right)\label{eq:C2-8}
\end{equation}
where $\varphi$ is a fixed continuous function. 
\end{lem}

\begin{rem}
The special case of (\ref{eq:C2-8}) corresponding to $\mu=\lambda=$
Lebesgue measure is 
\begin{equation}
\frac{\partial u}{\partial t}=\frac{1}{2}\frac{\partial^{2}u}{\partial x^{2}},\quad u\left(0,x\right)=\varphi\left(x\right).\label{eq:C2-9}
\end{equation}
 
\end{rem}

\begin{proof}
We shall refer to the literature, e.g., \cite{MR4106884,MR4020693,MR3687240,MR3402823,doi:10.1142/11980}.
Suffice it to say that \emph{the Ito-lemma} for the Gaussian process
$\{W_{t}^{\left(\mu\right)}\}$ is a key tool; together with the following
fact for the quadratic variation: 
\begin{equation}
\left(dW_{t}^{\left(\mu\right)}\right)^{2}=\mu\left(dt\right).\label{eq:C2-10}
\end{equation}
\end{proof}
Set $J=\left[0,1\right]$. Fix $\mu$, $\sigma$-finite and non-atomic.
Let $\lambda=dx=$ Lebesgue measure, and consider the following diagram
\[
\xymatrix{\mathscr{L}^{2}\left(\mu\right)\ar@/^{1.3pc}/[rr]^{T_{\mu}} &  & L^{2}\left(\mu\right)\ar@/^{1.3pc}/[ll]^{D=-d/dx}}
\]
with 
\[
\mathscr{L}^{2}\left(\mu\right):=\overline{dom\left(T_{\mu}\right)}^{L^{2}\left(\lambda\right)}.
\]
Let $K_{F}$ be as in (\ref{eq:C2-7}), then 
\begin{equation}
K_{F}\subseteq T_{\mu}T_{\mu}^{*},\label{eq:Cc-2}
\end{equation}
both are operators in $L^{2}\left(\mu\right)$. But $T_{\mu}T_{\mu}^{*}$
is \emph{selfadjoint} by general theory (see \secref{DP}), and restriction
of symmetric is symmetric.
\begin{thm}
Consider a fixed positive non-atomic Borel measure $\mu$ on $J=\left[0,b\right]$,
$b<\infty$; and define an operator $A$ on $L^{2}\left(\mu\right)=L^{2}\left(J,\mu\right)$
as follows: 

For $\varphi\in L^{2}\left(\mu\right)\cap C$, set 
\begin{equation}
\left(A\varphi\right)\left(x\right)=\int_{0}^{x}\left(\int_{0}^{y}\varphi\left(s\right)\mu\left(ds\right)\right)dy,\label{eq:Int1}
\end{equation}
then 
\begin{equation}
K_{F}A\varphi=\varphi.\label{eq:Int2}
\end{equation}
\end{thm}

\begin{proof}
Since $L^{2}\left(\mu\right)\subset L^{1}\left(\mu\right)$ the function
$y\longmapsto\int_{0}^{y}\varphi\left(s\right)\mu\left(ds\right)$
is continuous, and so $x\longmapsto\left(A\varphi\right)\left(x\right)$
is $C^{1}$ (one time differentiable with $\left(A\varphi\right)'\in C\left(J\right)$.)
Hence, for the LHS of (\ref{eq:Int2}) we have 
\[
K_{F}A:\varphi\longmapsto\nabla_{\mu}\frac{d}{dx}A\varphi=\nabla_{\mu}\int_{0}^{x}\varphi\left(s\right)\mu\left(ds\right)=\varphi\left(x\right),
\]
which is the desired conclusion (\ref{eq:Int2}).
\end{proof}
\begin{cor}
\label{cor:cpt1}Let the measure $\mu$ be specified as above, and
set 
\[
\left(A\varphi\right)\left(x\right)=\int_{0}^{x}\left(\int_{0}^{y}\varphi\left(s\right)\mu\left(ds\right)\right)dy
\]
(see (\ref{eq:Int1})), which defines a compact integral operator
in $L^{2}\left(\mu\right)$. 

Then $A=A_{\mu}$ is bounded and compact with triangular integral
kernel 
\begin{equation}
a\left(x,s\right)=\chi_{\left[0,x\right]}\left(s\right)\left(x-s\right).\label{eq:Int2a}
\end{equation}
\end{cor}

\begin{proof}
Without loss of generality we may work with $b=1$ and real Hilbert
space. An easy application of Schwarz to $L^{2}\left(\mu\right)$
shows that $A:L^{2}\left(\mu\right)\rightarrow L^{2}\left(\mu\right)$
is bounded. We have 
\begin{equation}
\left\langle A\varphi,\psi\right\rangle _{L^{2}\left(\mu\right)}=\int_{0}^{1}\int_{0}^{x}\left(x-s\right)\varphi\left(s\right)\psi\left(x\right)\mu\left(ds\right)\mu\left(dx\right).\label{eq:Int3}
\end{equation}
The asserted symmetry follows from this, or equivalently,
\begin{equation}
\left(A\varphi\right)\left(x\right)=\int_{0}^{1}a\left(x,s\right)\varphi\left(s\right)\mu\left(ds\right)\label{eq:Int4}
\end{equation}
where $a\left(x,s\right)=\chi_{\left[0,x\right]}\left(s\right)\left(x-s\right)$,
see \figref{sym}.

To see that the operator $A$ of (\ref{eq:Int1}) is compact as an
operator in $L^{2}(\mu)$, we make use of (\ref{eq:Int3}) and (\ref{eq:Int4})
as follows: We recall the fact the compact operators are the norm-closure
of finite rank operators. We then create such norm-limits of finite
rank operators with the use of the kernel (\ref{eq:Int2a}), and a
choice of a filter of partitions $P$ with disjoint Borel subsets
of the support of $\mu$. For each such partition $P$, we form rank-one
operators from the corresponding indicator functions from pairs of
sets $B,B'$ in $P$, and we then form the associated span of the
rank-one operators $\left|\chi_{B}\left\rangle \right\langle \chi_{B'}\right|$
by evaluation of (\ref{eq:Int2a}) with sample points chosen from
the partition sets. (We use Dirac's terminology $\left|\cdot\left\rangle \right\langle \cdot\right|$
for rank-one operators.) The Borel sets $B$, making up partitions,
are chosen with $\mu(B)$ finite, so the corresponding indicator functions
are in $L^{2}(\mu)$. For a fixed partition, we then form pairs of
such indicator functions, and the corresponding rank-one operators.
Since $\mu$ is chosen non-atomic, the partition-limit refinements
can be constructed such that the limit of the corresponding numbers
$\mu(B)$ is zero.
\end{proof}
\begin{figure}[H]
\includegraphics[width=0.35\textwidth]{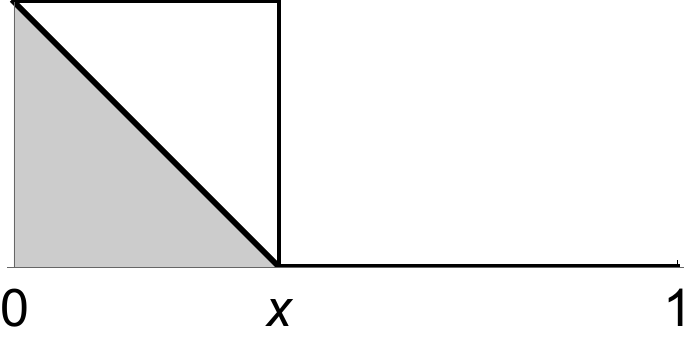}

\caption{\label{fig:sym}The kernel $a\left(x,s\right)=\chi_{\left[0,x\right]}\left(s\right)\left(x-s\right)=\left[x-s\right]_{+}=\max\left(0,x-s\right)$.}
\end{figure}

The following result gives a spectral decomposition for the Krein-Feller
operator $K_{F}$.
\begin{cor}
\label{cor:kek}Let $J=\left[0,1\right]$, and $\mu$ be a non-atomic
positive Borel measure on $J$. Set $g\left(x\right)=\mu\left(\left[0,x\right]\right)$. 

Consider $K_{F}=\frac{d}{d\mu}\frac{d}{dx}$ as a selfadjoint operator
in $L^{2}\left(\mu\right)$ with Neumann boundary condition, i.e.,
\[
dom\left(K_{F}\right)=\left\{ \psi\left(x\right)=\psi\left(0\right)+\int_{0}^{x}\left(\int_{0}^{y}fd\mu\right)dx:f\in L^{2}\left(\mu\right),\;\psi'\left(0\right)=\psi'\left(1\right)=0\right\} .
\]
Let $V:L^{2}\left(\mu\right)\rightarrow L^{2}\left(\mu\right)$ be
the integral operator defined as 
\begin{equation}
Vf\left(x\right)=\int_{0}^{1}H\left(x,t\right)f\left(t\right)dt\label{eq:it1}
\end{equation}
where 
\begin{equation}
H\left(x,t\right)=\begin{cases}
\left(g\left(x\right)-1\right)g\left(t\right) & t\leq x,\\
\left(g\left(t\right)-1\right)g\left(x\right) & x\geq t.
\end{cases}\label{eq:it2}
\end{equation}

Then $V$ is compact and selfadjoint, and we have the following eigenvalue
correspondence: 
\begin{equation}
c\in sp\left(K_{F}\right)\Longleftrightarrow c^{-1}\in sp\left(V\right).\label{eq:it3}
\end{equation}
\end{cor}

\begin{proof}
Selfadjointness of $V$ follows from (\ref{eq:it2}), and the argument
for compactness is the same as that of \corref{cpt1}. 

Let $\psi\left(x\right)=\psi\left(0\right)+\int_{0}^{x}\left(\int_{0}^{t}fd\mu\right)dt\in dom\left(K_{F}\right)$,
where $f\in L^{2}\left(\mu\right)$, so that $K_{F}\psi\left(x\right)=f\left(x\right)$.
Assume that 
\[
K_{F}\psi=c\psi,
\]
for some constant $c<0$. Then, 
\begin{eqnarray*}
K_{F}\psi & = & c\psi\\
 & \Updownarrow\\
f\left(x\right) & = & c\int_{0}^{x}\left(\int_{0}^{t}fd\mu\right)dt\\
 & \Downarrow\\
f'\left(x\right) & = & c\int_{0}^{x}fd\mu.
\end{eqnarray*}
The last line can be written as follows: 

Set $g\left(x\right):=\mu\left(\left[0,x\right]\right)$, then 
\begin{align}
f'\left(x\right) & =c\int_{0}^{x}fd\mu\nonumber \\
 & =c\int_{0}^{x}fdg\nonumber \\
 & =c\left(f\left(x\right)g\left(x\right)-\int_{0}^{x}g\left(t\right)f'\left(t\right)dt\right)\nonumber \\
 & =c\left(\left(f\left(0\right)+\int_{0}^{x}f'\left(s\right)ds\right)g\left(x\right)-\int_{0}^{x}g\left(t\right)f'\left(t\right)dt\right)\nonumber \\
 & =cf\left(0\right)g\left(x\right)+c\int_{0}^{x}\left(g\left(x\right)-g\left(t\right)\right)f'\left(t\right)dt.\label{eq:temp}
\end{align}
Since $f'=c\psi'$, the boundary condition $\psi'\left(1\right)=0$
and (\ref{eq:temp}) imply that 
\begin{equation}
cf\left(0\right)+c\int_{0}^{1}\left(1-g\left(t\right)\right)f'\left(t\right)dt=0.\label{eq:it4}
\end{equation}
Substitute (\ref{eq:it4}) into (\ref{eq:temp}), then 
\begin{align*}
f'\left(x\right) & =c\int_{0}^{x}\left(g\left(x\right)-g\left(t\right)\right)f'\left(t\right)dt+cf\left(0\right)g\left(x\right)\\
 & =c\left(\int_{0}^{x}\left(g\left(x\right)-g\left(t\right)\right)f'\left(t\right)dt-g\left(x\right)\int_{0}^{1}\left(1-g\left(t\right)\right)f'\left(t\right)dt\right)\\
 & =c\left(\int_{0}^{x}\left(g\left(x\right)-1\right)g\left(t\right)f'\left(t\right)dt+\int_{x}^{1}g\left(x\right)\left(g\left(t\right)-1\right)f'\left(t\right)dt\right)\\
 & =c\int_{0}^{1}H\left(x,t\right)f'\left(t\right)dt
\end{align*}
with $H$ as defined in (\ref{eq:it2}), and the assertion (\ref{eq:it3})
follows.
\end{proof}

\subsection{Path-space and Markov transition}

It is also of general interest to relate $K_{F}$ directly to the
generator of the diffusion semigroup. Notation: $\left(\Omega,\mathscr{F},\mathbb{P}\right)$,
$\Omega$ path space, $\mathscr{F}$ cylinder $\sigma$-algebra, $\mathbb{P}$
probability measure, $K_{\mu}\left(A\cap B\right)=\mu\left(A\cap B\right)$,
\begin{equation}
W_{t}^{\left(\mu\right)}\left(\omega\right)=\omega\left(t\right),\quad\forall\omega\in\Omega.\label{eq:Cc-3}
\end{equation}
\begin{align}
\mathbb{E}\left(\cdot\right) & =\int_{\Omega}\cdots d\mathbb{P}\quad\left(x\in J\right)\\
\mathbb{E}_{x}\left(\cdot\right) & =\int_{\Omega_{x}}\cdots d\mathbb{P},\quad\omega\in\Omega_{x}=\left\{ \omega,\omega\left(0\right)=x\right\} =\left\{ \omega,W_{0}^{\left(\mu\right)}\omega=x\right\} .\label{eq:CC-7}
\end{align}
In the discussion below we omit $\mu$ in $W^{\left(\mu\right)}$
to simplify notation.
\begin{equation}
\left(S_{t}\varphi\right)\left(x\right):=\mathbb{E}_{x}\left(\varphi\circ W_{t}^{\left(\mu\right)}\right).\label{eq:CC-8}
\end{equation}

We showed that $K_{F}$ is the generator of $S_{t}$ in (\ref{eq:CC-8}).
It is known that $S_{t}$ is a semigroup $\left(t\in\mathbb{R}_{+}\right)$
also $S_{0}=I$, so 
\[
S_{s}S_{t}=S_{s+t},\quad\forall s,t\in\mathbb{R}_{+}.
\]
For semigroups and generators in the Hilbert space framework, see
e.g., \cite{MR2849757,MR0231220,MR0133686}, and for diffusion semigroups,
we refer to e.g., \cite{MR126625,MR4243820,MR3571410}. Also see \cite[ch 3]{MR613983,MR1242198}.
\begin{lem}
\label{lem:Scont}$S_{t}$ is selfadjoint in $L^{2}\left(\mu\right)$
$\forall t\in\mathbb{R}_{+}$, so 
\begin{equation}
\int\left(S_{t}\varphi\right)\left(x\right)\psi\left(x\right)\mu\left(dx\right)=\int\varphi\left(x\right)\left(S_{t}\psi\right)\left(x\right)\mu\left(dx\right),\quad\forall\varphi,\psi\in L^{2}\left(\mu\right),\label{eq:Cc-8}
\end{equation}
also 
\begin{equation}
\int\left|S_{t}\varphi\right|^{2}d\mu\leq\int\left|\varphi\right|^{2}d\mu.\label{eq:Cc-9}
\end{equation}
\end{lem}

\begin{proof}
The proof of the properties (\ref{eq:Cc-8})--(\ref{eq:Cc-9}) is
contained in the literature of diffusion semigroups. But the following
proof sketch for (\ref{eq:Cc-8}) is new: 

Fix $t>0$ and any $0\leq s\leq t$, then
\begin{equation}
\frac{d}{ds}\int\left(S_{t-s}\varphi\right)\left(x\right)\left(S_{s}\psi\right)\left(x\right)\mu\left(dx\right)\equiv0\label{eq:Cc-10}
\end{equation}
so $s\rightarrow\int\left(S_{t-s}\varphi\right)\left(x\right)\left(S_{s}\psi\right)\left(x\right)\mu\left(dx\right)$
is constant value at $s=0=$ value at $s=t$, and (\ref{eq:Cc-8})
follows.

Proof of (\ref{eq:Cc-10}).
\begin{align*}
\text{LHS}_{\left(\ref{eq:Cc-10}\right)} & =\frac{d}{ds}\left\langle S_{t-s}\varphi,S_{s}\psi\right\rangle _{L^{2}\left(\mu\right)}\\
 & =-\left\langle K_{F}S_{t-s}\varphi,S_{s}\psi\right\rangle _{L^{2}\left(\mu\right)}+\left\langle S_{t-s}\varphi,K_{F}S_{s}\psi\right\rangle _{L^{2}\left(\mu\right)}=0
\end{align*}
since $K_{F}$ is symmetric w.r.t. $L^{2}\left(\mu\right)$.
\end{proof}

\subsection{The conditioning $W_{0}^{\left(\mu\right)}=x$}

For our considerations in (\ref{eq:CC-7}) and (\ref{eq:CC-8}) we
used the notation $\mathbb{E}_{x}$ and $\Omega_{x}$ with reference
to conditioning paths $\omega$ which ``start'' at $x$, so $\omega\left(0\right)=x$.
The justification is as follows. We have selected the sample space
$\Omega$ to be $\prod_{[0,\infty)}\mathbb{R}$ (Cartesian product),
and functions $\omega:\mathbb{R}_{\geq0}\rightarrow\mathbb{R}$ (infinitely
many ``paths''.) (It is known that the continuous functions will
have full measure relative to $\left(\Omega,\mathscr{C},\mathbb{P}\right)$
where $\mathscr{C}=$ the usual cylinder $\sigma$-algebra of subsets
of $\Omega$.) Here
\begin{equation}
\Omega_{x}:=\left\{ \omega\in\Omega:\omega\left(0\right)=x\right\} ,\;\text{and}\label{eq:CD-1}
\end{equation}
$\mathbb{P}$ denotes the probability measure on $\left(\Omega,\mathscr{C}\right)$,
such that 
\begin{equation}
\mathbb{E}\left(\cdot\cdot\right)=\int_{\Omega}\cdot\cdot d\mathbb{P},\;\text{and}\label{eq:CD-2}
\end{equation}
\begin{equation}
\mathbb{E}\left(W_{A}^{\left(\mu\right)}\right)=0,\quad\mathbb{E}\left(W_{A}^{\left(\mu\right)}W_{B}^{\left(\mu\right)}\right)=\mu\left(A\cap B\right),\label{eq:CD3}
\end{equation}
for all Borel sets $A,B\in\mathscr{C}$. 

From the construction the projection $\pi_{0}:\Omega\rightarrow\Omega$,
$\pi_{0}\left(\omega\right)=\omega\left(0\right)$, satisfies 
\begin{equation}
\mathbb{P}\circ\pi_{0}^{-1}\ll\mu.\label{eq:CD4}
\end{equation}
Now consider the Radon-Nikodym derivative: 
\begin{equation}
\frac{d\mathbb{P}\circ\pi_{0}^{-1}}{d\mu}=\mathbb{E}_{x},\label{eq:CD6}
\end{equation}
or equivalently, for all random variables $F$ on $\left(\Omega,\mathscr{C}\right)$
we have:
\begin{equation}
\mathbb{E}\left(F\right)=\int_{\mathbb{R}}\mathbb{E}_{x}\left(F\right)\mu\left(dx\right).\label{eq:CD7}
\end{equation}

\subsection{The $\mu$-heat equation}

We assume a fixed non atomic Borel measure $\mu$ supported in an
interval $J=\left[0,\alpha\right]$ where $\alpha$ may be finite
or infinite. We shall denote by $W^{\left(\mu\right)}$ the corresponding
\emph{generalized Brownian motion}, i.e., determined by: $W^{\left(\mu\right)}$
is Gaussian, real-valued 
\begin{equation}
\mathbb{E}(W^{\left(\mu\right)})=0,\quad\mathbb{E}(W_{A}^{\left(\mu\right)}W_{B}^{\left(\mu\right)})=\mu\left(A\cap B\right)\label{eq:C3-1}
\end{equation}
for all Borel sets $A,B\subset J$. 

For every continuous function $\varphi$ on $\mathbb{R}$, we consider
$\varphi(W_{A}^{\left(\mu\right)})=\varphi\circ W_{A}^{\left(\mu\right)}$.
If $A=\left[s,t\right]$ we make a choice of $W_{t}^{\left(\mu\right)}$
such that $W_{\left[s,t\right]}^{\left(\mu\right)}=W_{t}^{\left(\mu\right)}-W_{s}^{\left(\mu\right)}$,
and we set 
\begin{equation}
S_{t}^{\left(\mu\right)}\varphi\left(x\right)=\mathbb{E}_{x}\left(\varphi\circ W_{t}^{\left(\mu\right)}\right),\quad t\in\mathbb{R}_{+}.\label{eq:C3-2}
\end{equation}
Since, by (\ref{eq:C3-1}), the process $W^{\left(\mu\right)}$ has
independent increments, it follows that $S_{t}^{\left(\mu\right)}$
is a Markov semigroup.

When $\varphi$ is given, we set 
\begin{equation}
u\left(t,x\right)=\left(S_{t}^{\left(\mu\right)}\varphi\right)\left(x\right)=\mathbb{E}_{x}\left(\varphi\left(W_{t}^{\left(\mu\right)}\right)\right)\label{eq:C3-3}
\end{equation}
where the conditional expectation $\mathbb{E}_{x}$ corresponds to
$W_{0}^{\left(\mu\right)}=x$. Then the probability space $\Omega$
consists of continuous $\omega$, and 
\begin{equation}
W_{t}^{\left(\mu\right)}\left(\omega\right)=\omega\left(t\right),\quad0\leq t\leq\infty.\label{eq:C3-4}
\end{equation}
We then get the boundary condition $u\left(0,x\right)=\varphi\left(x\right)$
directly from (\ref{eq:C3-3}).

We shall further consider the operator $d/d\mu$ acting in the $t$-variable.
For convenience we shall write $\nabla_{t}^{\left(\mu\right)}$. 

We now turn to the corresponding diffusion equation:
\begin{thm}
\label{thm:Cc16}Let $\mu$, $W^{\left(\mu\right)}$, $S_{t}^{\left(\mu\right)}$,
and $u\left(t,x\right)$ be as specified. We then have 
\begin{equation}
\nabla_{t}^{\left(\mu\right)}u=\frac{1}{2}\frac{\partial^{2}}{\partial x^{2}}u.\label{eq:C3-5}
\end{equation}
\end{thm}

\begin{proof}
Without loss of generality we may assume $\varphi\in C^{2}$. Then
by Ito's lemma, we get 
\begin{equation}
d\varphi\left(W_{t}^{\left(\mu\right)}\right)=\varphi'\left(W_{t}^{\left(\mu\right)}\right)dW_{t}^{\left(\mu\right)}+\frac{1}{2}\varphi''\left(W_{t}^{\left(\mu\right)}\right)d\mu\label{eq:C3-6}
\end{equation}
where $\varphi''=\left(d/dx\right)^{2}\varphi$. For the derivation
of (\ref{eq:C3-6}), we refer to the cited papers (e.g., \cite{MR0345224,MR3571410,MR1242198,MR3687240}),
we also use the familiar quadratic variation formula 
\begin{equation}
QV=\left(dW_{t}^{\left(\mu\right)}\right)^{2}=\mu\left(dt\right).\label{eq:C3-7}
\end{equation}
Note that (\ref{eq:C3-7}) holds since $\mu$ was assumed non-atomic.
Further note that (\ref{eq:C3-6}) refers to \emph{Ito-differentials}.
In general (\ref{eq:C3-6}) is equivalent to the corresponding integral
formula version: 
\begin{equation}
\varphi\left(W_{t}^{\left(\mu\right)}\right)-\varphi\left(W_{0}^{\left(\mu\right)}\right)=\underset{\text{Ito-integral}}{\underbrace{\int_{0}^{t}\varphi'\left(W_{s}^{\left(\mu\right)}\right)dW_{s}^{\left(\mu\right)}}}+\frac{1}{2}\int_{0}^{t}\varphi''\left(W_{s}^{\left(\mu\right)}\right)\mu\left(ds\right).
\end{equation}

Now apply the expectation $\mathbb{E}_{x}$ to both sides in (\ref{eq:C3-7})
we arrive at 
\begin{equation}
\mathbb{E}_{x}\left(\varphi\left(W_{t}^{\left(\mu\right)}\right)\right)-\varphi\left(x\right)=\frac{1}{2}\int_{0}^{t}\mathbb{E}_{x}\left(\varphi''\left(W_{s}^{\left(\mu\right)}\right)\right)\mu\left(ds\right),\label{eq:C3-8}
\end{equation}
or equivalently: 
\begin{equation}
S_{t}^{\left(\mu\right)}\varphi\left(x\right)-\varphi\left(x\right)=\frac{1}{2}\int_{0}^{t}\left(\frac{\partial^{2}}{\partial x^{2}}u\right)\left(s,x\right)\mu\left(ds\right).\label{eq:C3-9}
\end{equation}
From the definition of the operator $T_{\left(\text{in \ensuremath{t}}\right)}^{\left(\mu\right)}=\nabla_{t}^{\left(\mu\right)}$
(see \lemref{C1}), we therefore get 
\begin{equation}
\nabla_{t}^{\left(\mu\right)}u=\frac{1}{2}\frac{\partial^{2}}{\partial x^{2}}u,\label{eq:C3-10}
\end{equation}
which is the desired conclusion (\ref{eq:C3-5}) in the Theorem.
\end{proof}
\begin{prop}
Consider the heat equation
\begin{equation}
\frac{\partial}{\partial t}u\left(t,x\right)=K_{F}u\left(t,x\right),\quad\left(t,x\right)\in\mathbb{R}_{+}\times\left[0,1\right],\label{eq:h1}
\end{equation}
with $K_{F}=\frac{\partial}{\partial\mu}\frac{\partial}{\partial x}$
given a selfadjoint realization in $L^{2}\left(\left[0,1\right],\mu\right)$.
Then the corresponding solution to (\ref{eq:h1}) has the following
form: 
\[
u\left(t,x\right)=\sum_{1}^{\infty}e^{-t\lambda_{n}}k_{n}\left(x\right),
\]
where $\lambda_{n}$ are the eigenvalues of $K_{F}$ and $k_{n}$
are the corresponding eigenfunctions. 
\end{prop}

\begin{proof}
The argument is based on separation of the two variables $t$ and
$x$, and use of spectral data; but now with reference to $K_{F}$
and $\nabla_{\mu}$. 

For details about choices of selfadjoint realizations of $K_{F}$,
see \corref{kek}, as well as \remref{sac} below.

Details as follows: Set 
\begin{equation}
u\left(t,x\right)=h\left(t\right)k\left(x\right).\label{eq:h2}
\end{equation}
Substituting (\ref{eq:h2}) into (\ref{eq:h1}) leads to
\[
h'\left(t\right)k\left(x\right)=h\left(t\right)\nabla_{x}^{\mu}k'\left(x\right),
\]
so that 
\[
\frac{h'}{h}\left(t\right)=\frac{\nabla_{x}^{\mu}k'\left(x\right)}{k\left(x\right)}=\text{const}=-\lambda.
\]
Thus, $h\left(t\right)=\text{const}e^{-\lambda t}$, and $\lambda$
is specified by 
\begin{equation}
-\lambda k\left(x\right)=\nabla_{x}^{\mu}k'\left(x\right).\label{eq:h3}
\end{equation}
Note, the eigenvalue problem (\ref{eq:h3}) is equivalent to 
\begin{equation}
h'\left(t\right)\int_{0}^{x}k\left(y\right)\mu\left(dy\right)=h\left(t\right)\left(k'\left(x\right)-k'\left(0\right)\right).\label{eq:h4}
\end{equation}
\end{proof}
\begin{rem}
\label{rem:sac}Solutions to (\ref{eq:h3}) depend on choices of boundary
conditions, i.e., selfadjoint realizations of the Krein-Feller operator.
Two examples are included below: 
\begin{enumerate}
\item Dirichlet boundary. (For a related discussion, see also \corref{kek}
above.) Dirichlet conditions: $f\left(0\right)=f\left(1\right)=0$.
Specifically,
\begin{align*}
dom\left(K_{F}\right) & =\big\{ f\in L^{2}\left(\mu\right):f\left(x\right)=\int_{0}^{x}\left(f'\left(0\right)+\int_{0}^{y}\varphi\left(s\right)\mu\left(ds\right)\right)dy,\\
 & \quad\quad f\left(1\right)=0,\:\varphi\in L^{2}\left(\mu\right)\big\}.
\end{align*}
In this case, one has 
\[
\left(K_{F}^{-1}\varphi\right)\left(x\right)=\int_{0}^{1}K_{\text{Dirichlet}}\left(x,s\right)\varphi\left(s\right)\mu\left(ds\right),
\]
where 
\[
K_{\text{Dirichlet}}\left(x,s\right)=\begin{cases}
\left(x-1\right)s & s\leq x,\\
x\left(s-1\right) & s\geq x.
\end{cases}
\]
\item $f\left(0\right)=f'\left(1\right)=0$. That is, 
\begin{align*}
dom\left(K_{F}\right) & =\big\{ f\in L^{2}\left(\mu\right):f\left(x\right)=\int_{0}^{x}\left(f'\left(0\right)+\int_{0}^{y}\varphi\left(s\right)\mu\left(ds\right)\right)dy,\\
 & \quad\quad f'\left(1\right)=0,\:\varphi\in L^{2}\left(\mu\right)\big\}.
\end{align*}
Then, 
\[
\left(K_{F}^{-1}\varphi\right)\left(x\right)=-\int_{0}^{1}\min\left(x,s\right)\varphi\left(s\right)\mu\left(ds\right).
\]
\end{enumerate}
\end{rem}

\begin{rem}
Our analysis of $W^{\left(\mu\right)}$ and the associated semigroup
is related to what is often referred to as \textquotedblleft change
of time;\textquotedblright{} see e.g., \cite{MR3363697} and (\ref{eq:tc2})
in \lemref{tc}.
\end{rem}

\begin{rem}
Dym and McKean developed a version of Krein-Feller operators in a
context of what they call \textquotedblleft strings\textquotedblright ,
see e.g., \cite{MR0442564,MR0448523} and also \cite{MR0247667}.
In principle, there is the following dictionary: string = positive
measure $\mu$ on a finite interval. Reasoning: every positive measure
on an interval is a Stieltjes measure by a monotone function, say
$F$. In Dym \& McKean, the monotone function $F$ measures the accumulation
of mass as you move forward on the string, and $\mu=dF$ as a Stieltjes
measure. However, Dym \& McKean do not seem to distinguish their analysis
for the dichotomy: $\mu$ singular or not. Recall, for the $1/3$
Cantor measure, $\mu=dF$ where $F$ is the Devil\textquoteright s
staircase function; see \figref{devil}. 

The case when $\mu\ll$ Lebesgue is covered in many other places,
e.g., books and papers by Edward Nelson, e.g., \cite{MR1189386,MR0343816,MR0214150,MR161189}. 
\end{rem}

\begin{rem}[Summary of extension theory for unbounded operators]
There is a theory in the case of unbounded operators in Hilbert space,
see e.g., \cite{MR0231220,MR126625,MR2849757}. 

Here, we emphasize the correspondence between skew-adjoint operators,
and generators of unitary one-parameter groups. In the case when skew-adjoint
operators arise as operator extensions, then they are specified by
partial isometries. On the other hand, dissipative operators correspond
to generators of contraction semigroups; and dissipative extensions
are specified by partial contractions. 

Generators of unitary one-parameter groups are maximal skew-symmetric
extensions. Examples of maximal skew-symmetric extensions that might
not be skew-adjoint will be when one of the indices is 0, so the cases
$(0,m)$ or $(n,0)$. Generators of contraction semigroups are maximal
dissipative extensions (for details, see \cite{MR1009163}.) One can
have semigroup generators for the cases $(n,0)$. But there are other
semigroup generators.
\end{rem}

Consider the operator $\left(d/dx\right)^{2}$ in $L^{2}\left(\left[0,1\right]\right)$. 
\begin{itemize}
\item Two particular selfadjoint extensions: 
\begin{itemize}
\item Neumann: $f'\left(0\right)=f'\left(1\right)=0$
\item Dirichlet: $f\left(0\right)=f\left(1\right)=0$
\end{itemize}
\item Maximal dissipative extension (diffusion semigroups)
\end{itemize}
\begin{rem}[Diffusion paths for Brownian motion, and for generalized Brownian
motion]

Diffusion paths:
\[
\mathbb{E}_{x}\left(\varphi\left(W_{t}\right)\right)=\left(S_{t}\varphi\right)\left(x\right)
\]
where $W_{t}$ is the Brownian motion (and $W^{\left(\mu\right)}$
for the general case).

Two cases:
\begin{enumerate}
\item all sample paths $\omega$, $\omega\left(0\right)=x$
\item restricted sample paths; e.g., $\omega\left(t\right)\in\left[0,1\right]$
or $\omega\left(t\right)\in\left[-1,1\right]$. 
\end{enumerate}
\end{rem}

\begin{figure}
\includegraphics[width=0.35\columnwidth]{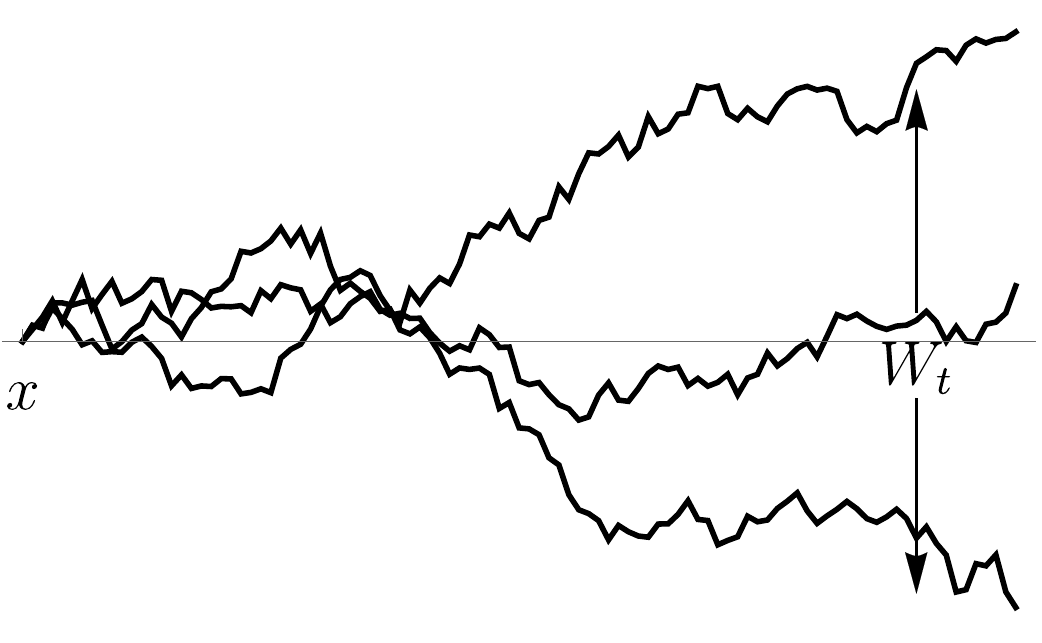}

\caption{}
\end{figure}

\begin{figure}
\includegraphics[width=0.35\columnwidth]{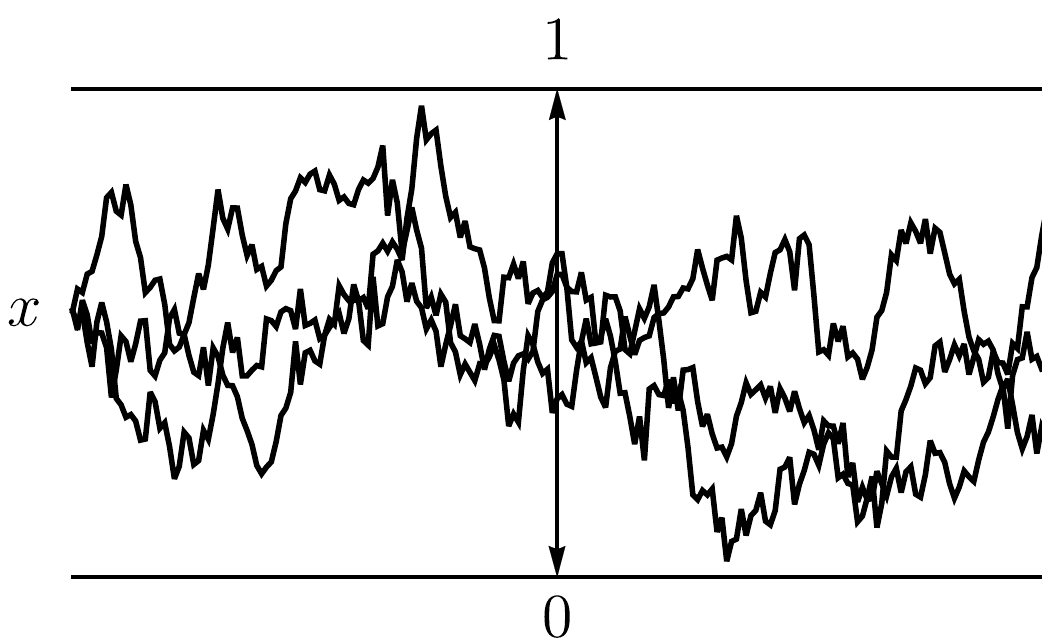}\qquad{}\includegraphics[width=0.35\columnwidth]{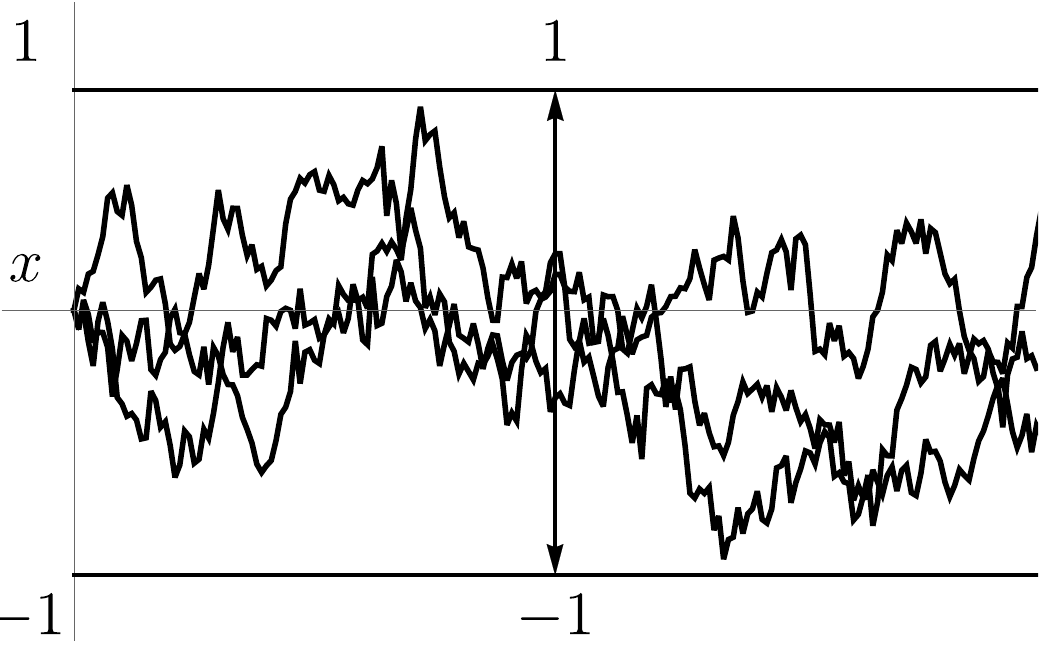}

\caption{}
\end{figure}

\subsection{Operators and generalized Dirichlet forms}
\begin{lem}
Let $\mu$ be a $\sigma$-finite positive measure as before, $\left(J,\mathscr{B}\right)$
with Borel $\sigma$-algebra $\mathscr{B}$. Here $J=[0,1)$ or $J=[0,\infty)$.
Then the following are equivalent:
\begin{enumerate}
\item \label{enu:t1}$f\left(y\right)-f\left(x\right)=\int_{x}^{y}T_{\mu}f\,d\mu$;
\item \label{enu:t2}$\int_{J}\varphi\,df=\int_{J}\varphi T_{\mu}f\,d\mu$,
$\forall\varphi\in C_{c}^{\infty}\left(J\right)$;
\item \label{enu:t3}$df\ll\mu$ and $\frac{df}{d\mu}=T_{\mu}f$. 
\end{enumerate}
If $\mathscr{H}\left(K_{\mu}\right)$ is the RKHS of the p.d. kernel
$K_{\mu}\left(A,B\right):=\mu\left(A\cap B\right)$, then in (\ref{enu:t3})
we have 
\begin{equation}
\left\Vert df\right\Vert _{\mathscr{H}\left(K_{\mu}\right)}=\left\Vert T_{\mu}f\right\Vert _{L^{2}\left(\mu\right)}.
\end{equation}

\end{lem}

\begin{proof}
In our previous paper (see e.g., \cite{2017arXiv170708492J,MR4025985})
we studied the RKHS $\mathscr{H}\left(K_{\mu}\right)$ as a Hilbert
space of measures $\rho$ such that $\rho\ll\mu$ and $\frac{d\rho}{d\mu}\in L^{2}\left(\mu\right)$,
$\left\Vert \rho\right\Vert _{\mathscr{H}\left(K_{\mu}\right)}=\left\Vert d\rho/d\mu\right\Vert _{L^{2}\left(\mu\right)}$;
and so we apply this result to the current setting, with $\rho=df$
as a Stieltjes measure. We will come back to this point in \secref{sc}.

Details: 

(\ref{enu:t1})$\Leftrightarrow$(\ref{enu:t2}). We have 
\begin{eqnarray}
\int\varphi df & \approx & \sum_{i}\varphi\left(x_{i}\right)\left(f\left(x_{i+1}\right)-f\left(x_{i}\right)\right)\label{eq:D2}\\
 & \underset{\text{by \ensuremath{\left(\ref{enu:t1}\right)}}}{=} & \sum_{i}\varphi\left(x_{i}\right)\int_{x_{i}}^{x_{i+1}}T_{\mu}f\,d\mu\nonumber \\
 & \approx & \int_{J}\varphi\left(T_{\mu}f\right)d\mu\quad\left(\text{standard integral approximation}\right).\nonumber 
\end{eqnarray}

(\ref{enu:t2})$\Rightarrow$(\ref{enu:t3}). Rewrite (\ref{enu:t2})
with $\varphi=\chi_{B}$, $B\in\mathscr{B}$, then 
\begin{equation}
\int_{B}df=\int_{B}\left(T_{\mu}f\right)d\mu.\label{eq:D3}
\end{equation}
Thus $\mu\left(B\right)=0$ $\Rightarrow$ $df\left(B\right)=0$,
and $df\ll\mu$ with $df/d\mu=T_{\mu}f$ which is (\ref{enu:t3}). 

(\ref{enu:t3})$\Rightarrow$(\ref{enu:t2}) is clear.
\end{proof}
\begin{cor}
\label{cor:D3}From \thmref{D2} we now get the following \uline{dual
pair} (with dense domains)
\[
T_{\mu}:\mathscr{L}^{2}\left(\mu\right)\longrightarrow L^{2}\left(\mu\right);
\]
and 
\[
D=-\frac{d}{dx}:L^{2}\left(\mu\right)\longrightarrow\mathscr{L}^{2}\left(\mu\right),
\]
where
\[
\mathscr{L}^{2}\left(\mu\right):=\overline{dom\left(T_{\mu}\right)}^{L^{2}\left(\lambda\right)},
\]
and $\lambda=d/dx=$ the usual Lebesgue measure restricted to the
fixed interval $J$. Recall, $\mu$ is assumed supported in $J$. 

We therefore obtain the following two selfadjoint operators: 
\begin{equation}
T_{\mu}T_{\mu}^{*}:L^{2}\left(\mu\right)\longrightarrow L^{2}\left(\mu\right);\label{eq:D5}
\end{equation}
and 
\begin{equation}
T_{\mu}^{*}T_{\mu}:\mathscr{L}^{2}\left(\mu\right)\longrightarrow\mathscr{L}^{2}\left(\mu\right)\label{eq:D6}
\end{equation}
with corresponding Dirichlet forms:
\begin{equation}
\left\langle \varphi,T_{\mu}T_{\mu}^{*}\varphi\right\rangle _{L^{2}\left(\mu\right)}=\int_{J}\left|\varphi'\right|^{2}d\mu;\label{eq:D7}
\end{equation}
and 
\begin{equation}
\left\langle f,T_{\mu}^{*}T_{\mu}f\right\rangle _{\mathscr{L}^{2}\left(\mu\right)}=\int_{J}|f^{\left(\mu\right)}|^{2}dx.\label{eq:D8}
\end{equation}
\end{cor}

For a given non-atomic measure $\mu$, we shall refer to the quadratic
form (\ref{eq:D7}) as the Dirichlet form induced from $\mu$. It
further follows from (\ref{eq:D5}) that this Dirichlet form has a
selfadjoint, semibounded $(A\geq0$) realization, say $A$, where
$A$ is a selfadjoint extension of our Krein-Feller operator $K_{F}$.
Here, initially $K_{F}$ is considered as a symmetric operator with
dense domain in $L^{2}(\mu)$. We shall further show that, when the
diffusion semigroup is realized in $L^{2}(\mu)$, then its infinitesimal
generator is $\lyxmathsym{\textendash}A$. It is known from general
theory that the Dirichlet form determines the diffusion semigroup;
and vice versa. See, e.g., \cite{MR613983,MR65809,MR87254,MR933819},
and also \subsecref{KFD}.

\section{\label{sec:rkhs}Applications to IFS measures }

The present section deals with applications of Krein-Feller operators
in the setting of IFS measures. For the operator theoretic framework,
see \secref{KF}, especially \corref{D3}. It is subdivided into two
subsections. The first subsection introduces a general class of iterated
function system (IFS) measures; while the second specializes to IFS
measures with support contained in finite intervals, and with the
Krein-Feller operators. Background references for this include \cite{MR625600,MR3838440,MR1655831,MR1667822,MR0030584,MR0047744,MR0228654,MR4184588,MR4025934,MR3110587,MR2793121}.

\subsection{\label{subsec:aifs}Iterated function system measures}

The theory of iterated function system (IFS) measures is extensive.
IFS measures arise in diverse applications, geometric analysis, fractal
harmonic analysis, chaotic dynamics and more. Here we wish to cite
the following papers of most direct relevance for our current discussion,
\cite{MR4025985,MR3996038,MR3968033,MR3958137,MR3793614,MR3736758,MR3559001,MR3552934,MR4211868,MR4149147,MR4080661,MR2793121,MR3838440}.
\begin{defn}
Let $\left(X,\mathscr{B}_{X}\right)$ be a measurable space, $N\in\mathbb{N}$,
and let $\left\{ \sigma_{i}\right\} _{i=1}^{N}$ be a system of continuous
endomorphisms $\sigma_{i}:X\rightarrow X$. Let $\left\{ p_{i}\right\} _{i=1}^{N}$,
$p_{i}>0$, $\sum_{i=1}^{N}p_{i}=1$ be fixed. 

A Borel measure $\mu$ on $X$ is said to be an iterate function system
(IFS) measure w.r.t. the data iff (Def) the following identity holds:
\begin{equation}
\mu=\sum_{i=1}^{N}p_{i}\,\mu\circ\sigma_{i}^{-1}\label{eq:Dd2}
\end{equation}
on the Borel $\sigma$-algebra $\mathscr{B}_{X}$. 

We now turn to an explicit realization of the IFS-measure from e.q.
(\ref{eq:Dd2}).
\end{defn}

\begin{thm}
\label{thm:D4}Let $X$, $N$, $\left\{ \sigma_{i}\right\} _{i=1}^{N}$,
$\left\{ p_{i}\right\} _{i=1}^{N}$ be as above. Consider the infinite
product 
\begin{equation}
\Omega_{N}:=\prod_{\mathbb{N}_{0}}\left\{ 1,2,\cdots,N\right\} ,\label{eq:Dd3}
\end{equation}
and suppose, for all $\omega\in\Omega_{N}$, the intersection below
is a singleton, i.e., 
\begin{equation}
\bigcap_{\stackrel{n=1}{\omega|_{n}=\left\{ i_{1},\cdots,i_{n}\right\} }}^{\infty}\sigma_{i_{n}}\sigma_{i_{n-1}}\cdots\sigma_{i_{1}}\left(X\right)=\left\{ x\left(\omega\right)\right\} ;\label{eq:Dd4}
\end{equation}
then there is an associated IFS measure $\mu$ constructed from the
infinite product
\begin{equation}
\pi:=p\times p\times p\times\cdots\;\text{on }\:\Omega_{N}.\label{eq:Dd5}
\end{equation}
Because of assumption (\ref{eq:Dd4}), we get a well-defined $X$-valued
random variable $W$ (for the probability space $\left(\Omega_{N},\pi\right)$),
and the IFS-measure $\mu$ from (\ref{eq:Dd2}) is then the distribution
of $W$. See \figref{subd}.
\end{thm}

\begin{figure}[H]
\includegraphics[width=0.9\columnwidth]{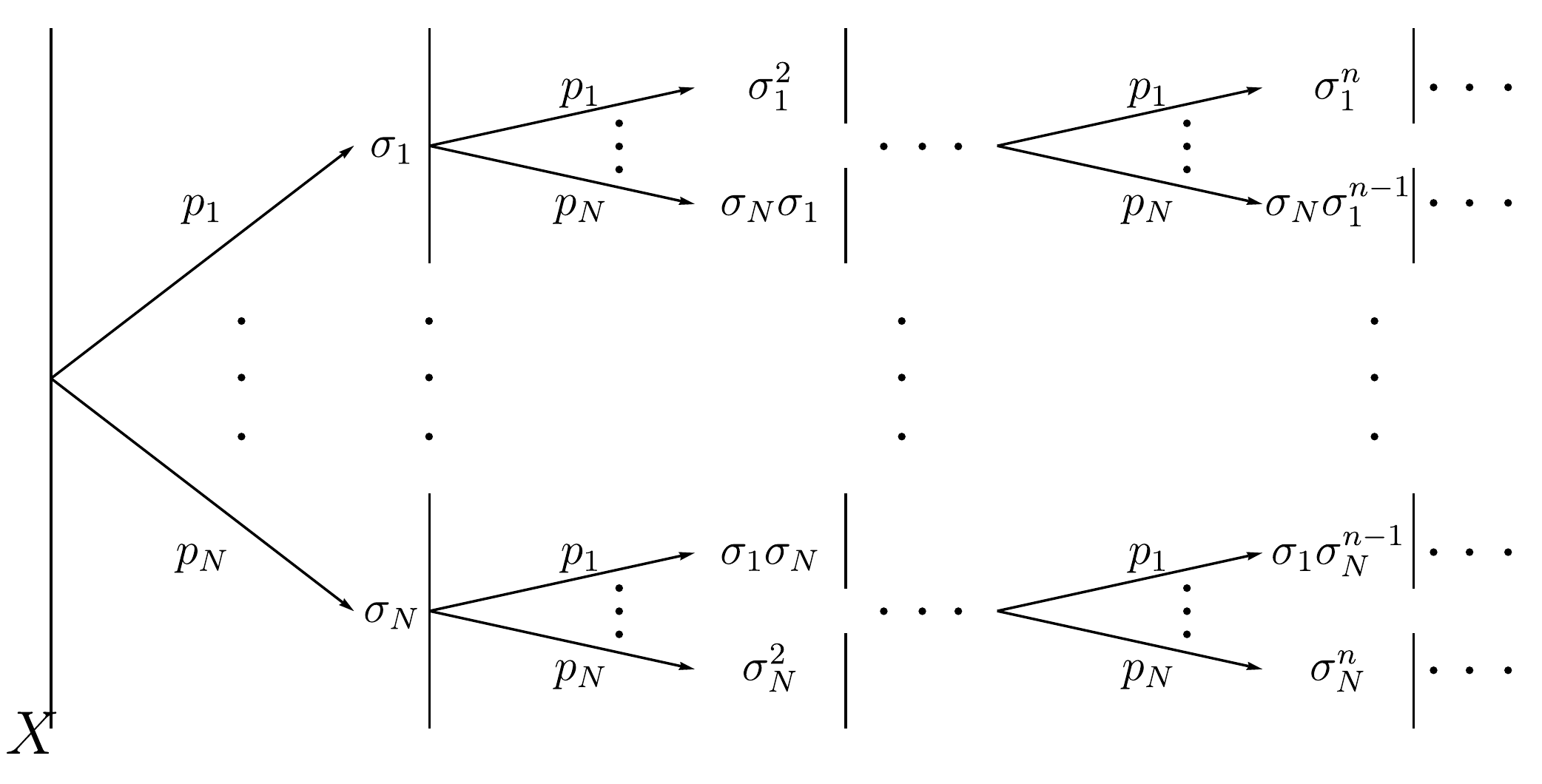}

\caption{\label{fig:subd}}

\end{figure}

\begin{proof}
With (\ref{eq:Dd4}), we define the random variable $W^{\left(IFS\right)}\left(\omega\right):=x\left(\omega\right)$,
$\omega\in\Omega_{N}$. If $\left(i_{1},\cdots,i_{n}\right)\in\prod_{1}^{n}\left\{ 1,\cdots,N\right\} $,
then the measure $\pi$ is specified on cylinder sets as follows:
\begin{equation}
\pi\left(\left[i_{1},\cdots,i_{n}\right]\right)=p_{i_{1}}p_{i_{2}}\cdots p_{i_{n}}.\label{eq:Dd6}
\end{equation}
The measure $\pi$ is then defined on $\Omega_{N}$ via Kolomogorov's
consistency extension theorem, see \cite{MR562914,MR735967,MR0494348}.
Let 
\begin{equation}
W:\Omega_{N}\longrightarrow X\label{eq:Dd7}
\end{equation}
be the random variable specified by the condition in (\ref{eq:Dd4}),
and set 
\begin{equation}
\mu:=\pi\circ W^{-1}.\label{eq:Dd8}
\end{equation}
One checks that $\mu$ will then satisfy the IFS condition in (\ref{eq:Dd2}).
\end{proof}
\begin{cor}
\label{cor:D5}In this corollary we fix a system $\left\{ \sigma_{i}\right\} _{i=1}^{N}$
of endomorphisms, and we consider the IFS measure $\mu$ as it depends
on the choice of probability weights $p=\left(p_{i}\right)_{i=1}^{N}$,
$\sum p_{i}=1$. Set $\mu^{\left(p\right)}=$ the solution to (\ref{eq:Dd2});
(see also (\ref{eq:Dd7})). Then if $p\neq q$ (i.e., $\exists i$
such that $p_{i}\neq q_{i}$) then the two measures $\mu^{\left(p\right)}$
and $\mu^{\left(q\right)}$ are mutually singular.
\end{cor}

\begin{proof}
The result is immediate from (\ref{eq:Dd8}) and Kakutani's dichotomy
theorem for infinite product measures, applied to (\ref{eq:Dd5}).
By Kakutani \cite{MR14404}, the two measures $\times_{\mathbb{N}}p$
and $\times_{\mathbb{N}}q$ are mutually singular; and by (\ref{eq:Dd8}),
so are the two IFS measures $\mu^{\left(p\right)}$ and $\mu^{\left(q\right)}$. 
\end{proof}

\subsection{IFS measures supported on compact intervals}

Here we take 
\begin{equation}
\sigma_{i}\left(x\right):=\lambda_{i}x+b_{i}\label{eq:Dd9}
\end{equation}
where $0<\lambda_{i}<1$, $b_{i}\in\mathbb{R}$, $x\in\mathbb{R}$,
$1\leq i\leq N$. Fix $\left\{ p_{i}\right\} _{i=1}^{N}$ as above.
Then the corresponding IFS measure $\mu$ (see (\ref{eq:Dd5})) will
satisfy 
\begin{equation}
\sum_{i=1}^{N}p_{i}\int f\left(\lambda_{i}x+b_{i}\right)\mu\left(dx\right)=\int f\left(x\right)\mu\left(dx\right)\label{eq:Dd10}
\end{equation}
for all bounded continuous functions $f$ on $\mathbb{R}$. One checks
that $\mu$ will then be supported on a compact interval $J\subset\mathbb{R}$. 

\renewcommand{\arraystretch}{2}
\begin{example}[Three IFS measures, Lebesgue measure and two Cantor measures]
\label{exa:DE1} Let $N=2$ and $\left\{ p_{i}\right\} =\left\{ \frac{1}{2},\frac{1}{2}\right\} $. 
\begin{enumerate}
\item \label{enu:de1}$\begin{cases}
\sigma_{1}\left(x\right)=\frac{x}{2}\\
\sigma_{2}\left(x\right)=\frac{x+1}{2}
\end{cases}$ $J=\left[0,1\right]$, $\mu=$ restricted Lebesgue measure. 
\item \label{enu:de2}$\begin{cases}
\sigma_{1}\left(x\right)=\frac{x}{3}\\
\sigma_{2}\left(x\right)=\frac{x+2}{3}
\end{cases}$ $J=\left[0,1\right]$, $\mu_{3}=$ middle $\frac{1}{3}$ Cantor measure;
Hausdorff dim = $\ln2/\ln3$. (See \figref{devil})
\item \label{enu:de3}$\begin{cases}
\sigma_{1}\left(x\right)=\frac{x}{4}\\
\sigma_{2}\left(x\right)=\frac{x+2}{4}
\end{cases}$ $J=\left[0,1\right]$, $\mu_{4}=$ Cantor measure with two omitted
intervals; Hausdorff dim = $\frac{1}{2}$.
\end{enumerate}
\renewcommand{\arraystretch}{1}
\end{example}

\begin{rem}
There is an important difference between the cases (\ref{enu:de2})
and (\ref{enu:de3}) above. Naturally they have different geometries,
different Hausdorff dimension, and they are mutually singular. They
are both IFS measures, but the most striking difference is their respective
harmonic analysis. For the middle fourth Cantor measure $\mu_{4}$
in (\ref{enu:de3}), the corresponding $L^{2}(\mu_{4})$ admits an
orthogonal Fourier series expansion; while the middle third Cantor
measure $\mu_{3}$ in (\ref{enu:de2}) does not. Even more striking
is the fact that $L^{2}(\mu_{3})$ does not admit three orthogonal
Fourier exponentials. For this subject, and related, readers are referred
to \cite{MR1667822,MR1655831,MR3838440}.
\end{rem}

\begin{cor}
Consider a measure $\mu$ specified as in (\ref{eq:Dd10}) above,
so it includes the cases (\ref{enu:de1})--(\ref{enu:de3}) in \exaref{DE1}.

Then for the Dirichlet form (\ref{eq:D7}) in \corref{D3} we have
\begin{equation}
\int_{J}\left|\varphi'\right|^{2}d\mu=\sum_{i=1}^{N}\frac{p_{i}}{\lambda_{i}^{2}}\int\left|\frac{d}{dx}\left(\varphi\left(\lambda_{i}x+b_{i}\right)\right)\right|^{2}\mu\left(dx\right).\label{eq:Dd12}
\end{equation}
\end{cor}

\textbf{Cumulative functions $g_{\mu}$ for general IFS measures $\mu$.
}Let $\mu$ be an IFS measure as in \thmref{D4}, and let 
\[
g_{\mu}\left(x\right):=\mu\left(\left[0,x\right]\right).
\]
The scale 3 Cantor measure $\mu_{3}$ with $g_{\mu_{3}}$ is discussed
in \remref{sm} and \exaref{DE1}, see also \figref{devil}. The scale
4 Cantor measure $\mu_{4}$ with $g_{\mu_{4}}$ from \exaref{DE1}
is shown in \figref{c4} below. Note these two cases both have equal
weights $p=\left(p_{i}\right)_{i=1}^{2}=\left\{ 1/2,1/2\right\} $.

\begin{figure}[H]
\includegraphics[width=0.35\columnwidth]{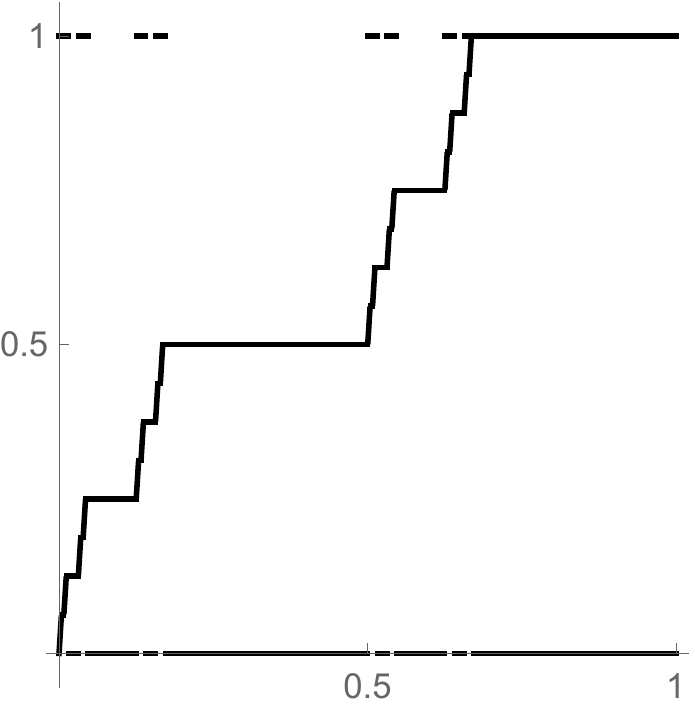}

\caption{\label{fig:c4}$g_{\mu_{4}}\left(\cdot\right)$ the \emph{scale 4
Cantor measure} as a Stieltjes measure.}
\end{figure}

Consider more general IFS measures on $\left[0,1\right]$, i.e., the
unique solutions $\mu$ to: 
\[
\mu=p_{1}\mu\circ\sigma_{1}^{-1}+p_{2}\mu\circ\sigma_{2}^{-1}.
\]
Define
\[
f_{0}\left(x\right)=x\chi_{\left[0,1\right]}\left(x\right)+\chi_{\left(1,\infty\right)}\left(x\right),
\]
and 
\begin{equation}
f_{n}\left(x\right):=p_{1}f_{n-1}\left(\sigma_{1}^{-1}\left(x\right)\right)+p_{2}f_{n-1}\left(\sigma_{2}^{-1}\left(x\right)\right),\quad n\in\mathbb{N}.
\end{equation}
Then 
\begin{equation}
\lim_{n\rightarrow\infty}f_{n}\left(x\right)=g_{\mu}\left(x\right),\quad\forall x\in\left[0,1\right].
\end{equation}

\begin{example}
\label{exa:Cagain}Let $p=\left(p_{i}\right)_{i=1}^{2}=\left\{ \frac{1}{3},\frac{2}{3}\right\} $.
An illustration of $g_{\mu}\left(\cdot\right)$ for the two cases
below is in \figref{lebesgue}.
\begin{enumerate}
\item \label{enu:case3} $\sigma_{1}\left(x\right)=\frac{x}{2}$, $\sigma_{2}\left(x\right)=\frac{x+1}{2};$
\item \label{enu:case4}$\sigma_{1}\left(x\right)=\frac{x}{3}$, $\sigma_{2}\left(x\right)=\frac{x+2}{3}$. 
\end{enumerate}
For case (\ref{enu:case3}), we have 
\[
\left\{ \left(a_{i}\right)\in\prod_{\mathbb{N}}\left\{ 0,1\right\} ,\sum_{1}^{\infty}\frac{a_{i}}{2^{i}}\leq x\right\} \xrightarrow{\;\mu\;}\lim_{m\rightarrow\infty}\left(\frac{2}{3}\right)^{m}\left(\frac{1}{2^{i_{1}}}+\frac{1}{2^{i_{2}}}+\cdots+\frac{1}{2^{i_{m}}}\right)=g_{\mu}\left(x\right).
\]

\end{example}

\begin{figure}[H]
\subfloat[\label{fig:c3varA}$\sigma_{1}\left(x\right)=\frac{x}{2}$, $\sigma_{2}\left(x\right)=\frac{x+1}{2}$.]{\includegraphics[width=0.45\columnwidth]{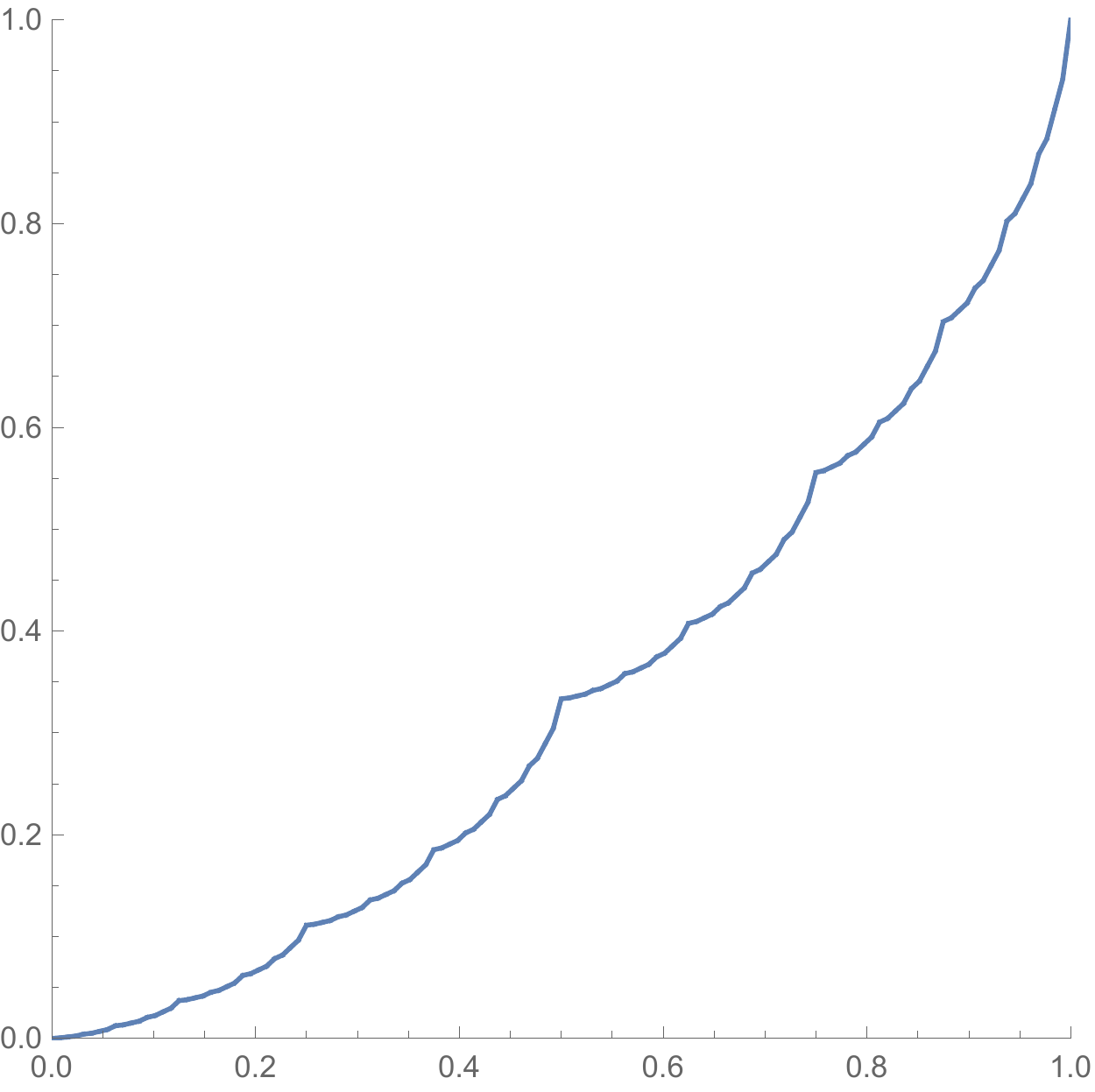}

}\hfill{}\subfloat[\label{fig:c3-var}$\sigma_{1}\left(x\right)=\frac{x}{3}$, $\sigma_{2}\left(x\right)=\frac{x+2}{3}$.]{\includegraphics[width=0.45\columnwidth]{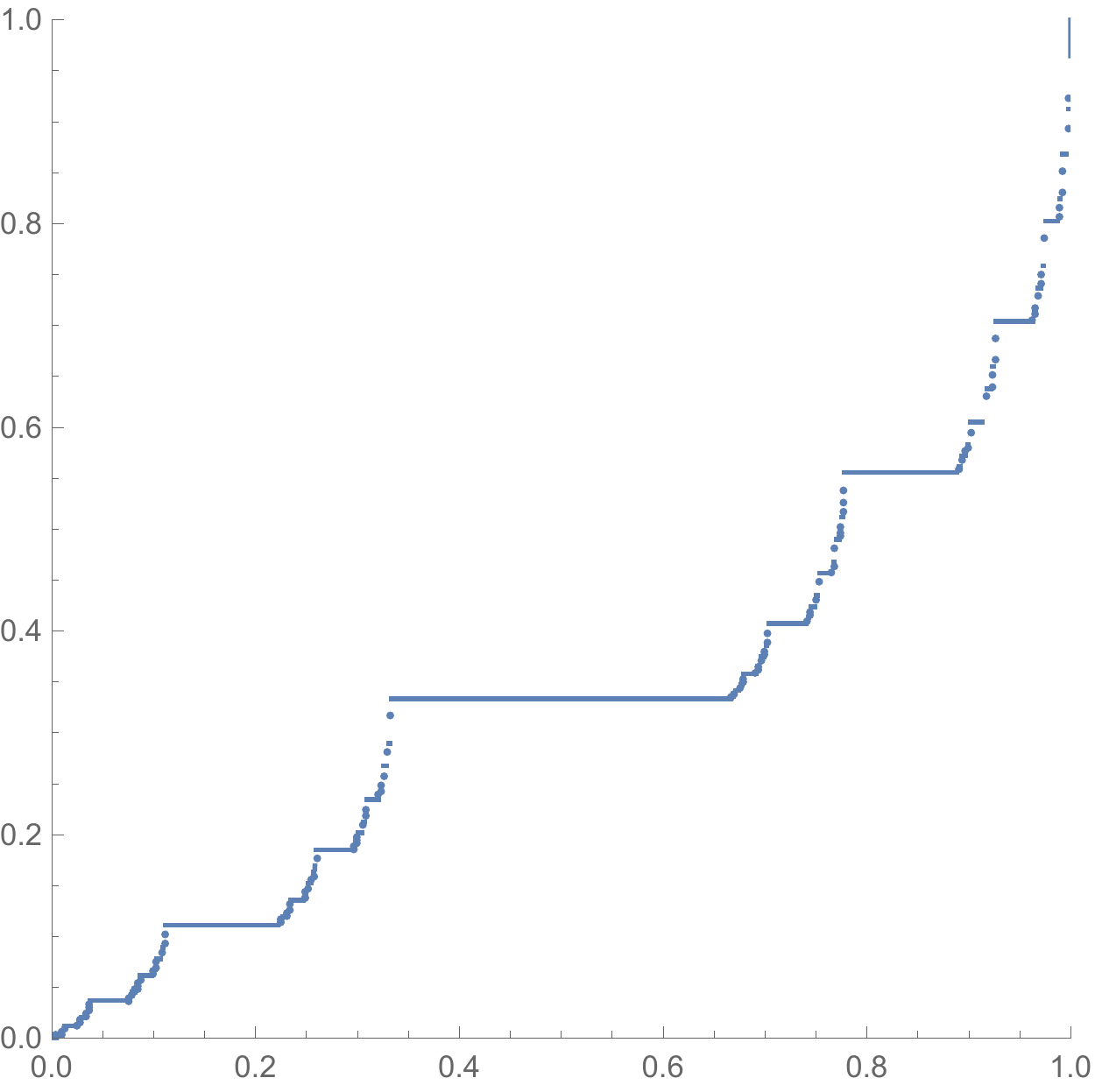}

}

\caption{\label{fig:lebesgue}$g_{\mu}\left(x\right)$ with $p=\left\{ \frac{1}{3},\frac{2}{3}\right\} $.}
\end{figure}

For the related IFS-measures and their cumulative distributions discussed
earlier, see \figref{devil} (\remref{sm}, scale-3 Cantor), \exaref{DE1},
\figref{c4} (scale-4 Cantor). For these cases, the fair-coin measure
is used. And by contrast, \figref{lebesgue} illustrates a choice
of biased Markov chain-weights.

In particular, it follows from \corref{D5} above that the measure
$\mu$ in \exaref{Cagain} (\ref{enu:case3}), see \figref{c3varA},
is mutually singular with respect to Lebesgue measure $\lambda$.
They are mutually singular despite the fact that both measures, $\mu$
and $\lambda$, on the unit-interval, arise from the same pair of
maps, $\left\{ \sigma_{i}\right\} $ by IFS-recursive iteration.

\section{\label{sec:sc}Special case (intervals) vs general measure spaces}

Observation: For general measure spaces $\left(X,\mathscr{B},\mu\right)$,
in an earlier paper, we established a canonical isometry $T_{\mu}$
of the $\text{RKHS}\left(K_{\mu}\right)$ onto $L^{2}\left(\mu\right)$.

In the special case of $X=$ an interval, and $\mathscr{B}=$ the
Borel sigma-algebra, $\mu$ a singular non-atomic measure, we also
have an operator $T_{\mu}$ and it is a special case of the $T_{\mu}$
we introduced in our earlier paper on RKHS theory. Background references
for this include \cite{MR3559001,MR885633,MR0214150,MR0343816,MR1189386,MR3687240,MR3838440,MR3996038}.

\begin{rem}[Distinction between first order and second order operators]
 Our KF-Laplacian (second order) has selfadjoint extensions, for
example $T^{*}T$. As we study the KF-Laplacian with minimal domain,
we study its selfadjoint extensions.
\end{rem}

Let $\left(X,\mathscr{B},\mu\right)$ be a $\sigma$-finite measure
space. We then consider the p.d. kernel on $\mathscr{B}_{fin}\times\mathscr{B}_{fin}$,
defined as 
\begin{equation}
K_{\mu}\left(A,B\right):=\mu\left(A\cap B\right),\quad A,B\in\mathscr{B}_{fin}\label{eq:F1}
\end{equation}
with $\mathscr{H}\left(K_{\mu}\right)$ being the associated RKHS. 
\begin{thm}
\label{thm:F1}We have 
\begin{align}
\mathscr{H}\left(K_{\mu}\right) & =\Big\{ F:F\text{ \ensuremath{\text{\ensuremath{\sigma}-finite measure on \ensuremath{\left(X,\mathscr{B}\right)}}} s.t. }\label{eq:F2}\\
 & \qquad F\ll\mu,\:T_{\mu}F:=dF/d\mu,\:\left\Vert F\right\Vert _{\mathscr{H}\left(K\right)}=\left\Vert dF/d\mu\right\Vert _{L^{2}\left(\mu\right)}\Big\}.\label{eq:F3}
\end{align}
\end{thm}

\begin{proof}
We also included proof details for the conclusions (\ref{eq:F2})-(\ref{eq:F3})
when $K=K_{\mu}$ is specified as in (\ref{eq:F1}).

Recall that 
\begin{equation}
\mu\left(\cdot\cap A\right)\in\mathscr{H}\left(K_{\mu}\right).\label{eq:F4}
\end{equation}
So if $F$ is a signed measure on $\left(X,\mathscr{B}\right)$ and
$F\in\mathscr{H}\left(K_{\mu}\right)$, then we assign the inner product
\begin{equation}
\left\langle F,\mu\left(\cdot\cap A\right)\right\rangle _{\mathscr{H}\left(K_{\mu}\right)}=F\left(A\right),\label{eq:F5}
\end{equation}
using the reproducing property of $\mathscr{H}\left(K_{\mu}\right)$.

From (\ref{eq:F5}), $F$ is a function on $\mathscr{B}$, and we
showed that from the axioms of the RKHS $\mathscr{H}\left(K_{\mu}\right)$
that $F\left(\cdot\right)$ will be $\sigma$-additive, so a signed
measure. Specifically, if $B=\cup_{i}B_{i}$, $B_{i}\in\mathscr{B}$,
$B_{i}\cap B_{j}=\emptyset$ for $i\neq j$, one has 
\begin{equation}
F\left(B\right)=\sum_{i}F\left(B_{i}\right).
\end{equation}

But we also derive the axioms for $\mathscr{H}\left(K_{\mu}\right)$
as follows: 
\begin{equation}
F\ll\mu\rightsquigarrow\frac{dF}{d\mu}=\text{Radon Nikodym derivative}
\end{equation}
In particular if $A\in\mathscr{B}_{fin}$ is fixed, then 
\begin{equation}
\mu\left(\cdot\cap A\right)\ll\mu\label{eq:F8}
\end{equation}
and
\begin{equation}
\frac{d\mu\left(\cdot\cap A\right)}{d\mu}=\chi_{A}\left(\cdot\right).\label{eq:F9}
\end{equation}
To see (\ref{eq:F9}), note that 
\[
\int_{B}\chi_{A}\left(\cdot\right)d\mu=\mu\left(A\cap B\right).
\]

Moreover, for all $F,G\in\mathscr{H}\left(K_{\mu}\right)$, we have
\begin{equation}
\left\langle F,G\right\rangle _{\mathscr{H}\left(K_{\mu}\right)}=\int_{X}\frac{dF}{d\mu}\frac{dG}{d\mu}d\mu.\label{eq:F10}
\end{equation}
The formula (\ref{eq:F10}) offers another way to verify (\ref{eq:F5}).
Indeed,
\begin{eqnarray*}
\text{LHS}_{\left(\ref{eq:F5}\right)} & \underset{\text{by \ensuremath{\left(\ref{eq:F10}\right)}}}{=} & \int_{X}\frac{dF}{d\mu}\left(\cdot\right)\frac{d\mu\left(\cdot\cap A\right)}{d\mu}d\mu\\
 & \underset{\text{by \ensuremath{\left(\ref{eq:F9}\right)}}}{=} & \int_{X}\frac{dF}{d\mu}\left(\cdot\right)\chi_{A}\left(\cdot\right)d\mu\\
 & = & \int_{A}\frac{dF}{d\mu}d\mu\\
 & = & F\left(A\right)=\text{RHS}_{\left(\ref{eq:F5}\right)}.
\end{eqnarray*}
\end{proof}
\textbf{Conclusion. }Fix $\left(X,\mathscr{B},\mu\right)$. Recall:
$\mathscr{H}\left(K_{\mu}\right)$ RKHS of $K_{\mu}$, consisting
of signed measures $F$ s.t. $F\ll\mu$ and $dF/d\mu\in L^{2}\left(\mu\right)$.

\begin{tabular}{>{\centering}p{0.45\columnwidth}>{\centering}p{0.45\columnwidth}}
$F\in\xymatrix{\mathscr{H}\left(K_{\mu}\right)\ar@/^{1.3pc}/^{T_{\mu}}[rr] &  & L^{2}\left(\mu\right)\ar@/^{1.3pc}/^{T_{\mu}^{*}}[ll]}
$ & $T_{\mu}F=dF/d\mu\in L^{2}\left(\mu\right)$\tabularnewline
$T_{\mu}^{*}T_{\mu}=I_{\mathscr{H}\left(K_{\mu}\right)}$, and $T_{\mu}T_{\mu}^{*}=I_{L^{2}\left(\mu\right)}$ & $\left(T_{\mu}^{*}\psi\right)\left(A\right)=\int_{A}\psi d\mu$, $\forall A\in\mathscr{B}$\tabularnewline
\end{tabular}

If $F$ is represented as a signed Stieltjes measure, $F=df$, then
$\nabla^{\left(\mu\right)}f=T_{\mu}F$. 

\renewcommand{\arraystretch}{2}

\begin{table}[H]
\begin{tabular}{|>{\centering}p{0.45\columnwidth}|>{\centering}p{0.45\columnwidth}|}
\hline 
\multicolumn{2}{|c|}{RKHS; different settings for $K_{\mu}\left(A,B\right)=\mu\left(A\cap B\right)$}\tabularnewline
\hline 
\hline 
$\left(X,\mathscr{B}\right)$ general, $\mu\left(A\cap B\right)$ & special measures on $\left(X,\mathscr{B}\right)$\tabularnewline
\hline 
\multirow{2}{0.45\columnwidth}{$X\subset\mathbb{R}$, so $\left[0,1\right]$ or $\left[0,\infty\right]$
etc. $A=\left[0,x\right]$, $B=\left[0,y\right]$, $K_{\mu}\left(A,B\right)=\mu\left(\left[0,x\wedge y\right]\right)$} & \multirow{2}{0.45\columnwidth}{Special case $\mu=\lambda=dx=$ Lebesgue measure, $K_{\lambda}=x\wedge y$}\tabularnewline
 & \tabularnewline
\hline 
\multicolumn{2}{|>{\raggedright}p{1\columnwidth}|}{GENERAL $F\in\mathscr{H}\left(K_{\mu}\right)$ with $\sigma$-finite
measure s.t. $F\ll\mu$, $T_{\mu}F=\frac{dF}{d\mu}$}\tabularnewline
\hline 
\multicolumn{2}{|>{\raggedright}p{0.9\columnwidth}|}{$X\subset\mathbb{R}$, $F=df$ where $f$ is a bounded variation function
on $\mathbb{R}$; Stieltjes measure $df\ll\mu$, $T_{\mu}F=\frac{dF}{d\mu}$}\tabularnewline
\hline 
\end{tabular}

\caption{}
\end{table}

\renewcommand{\arraystretch}{1}

\begin{figure}[H]
\subfloat[{$X=\mathbb{R}$, $B=\left[0,x\right]$, $A=\left[x,y\right]$}]{\includegraphics[width=0.45\columnwidth]{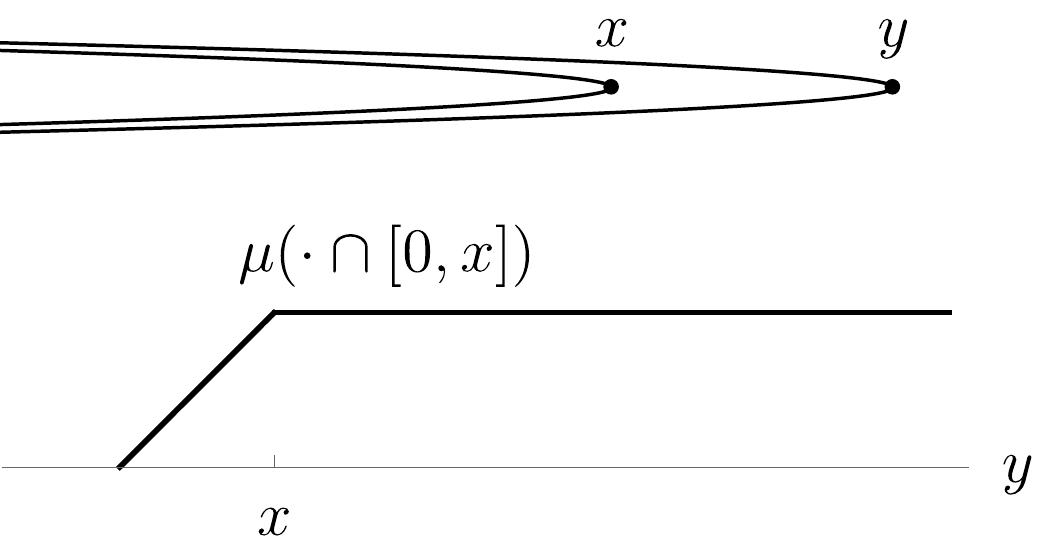}

}\subfloat[$\left(X,\mathscr{B},\mu\right)$ general measure space]{\includegraphics[width=0.45\columnwidth]{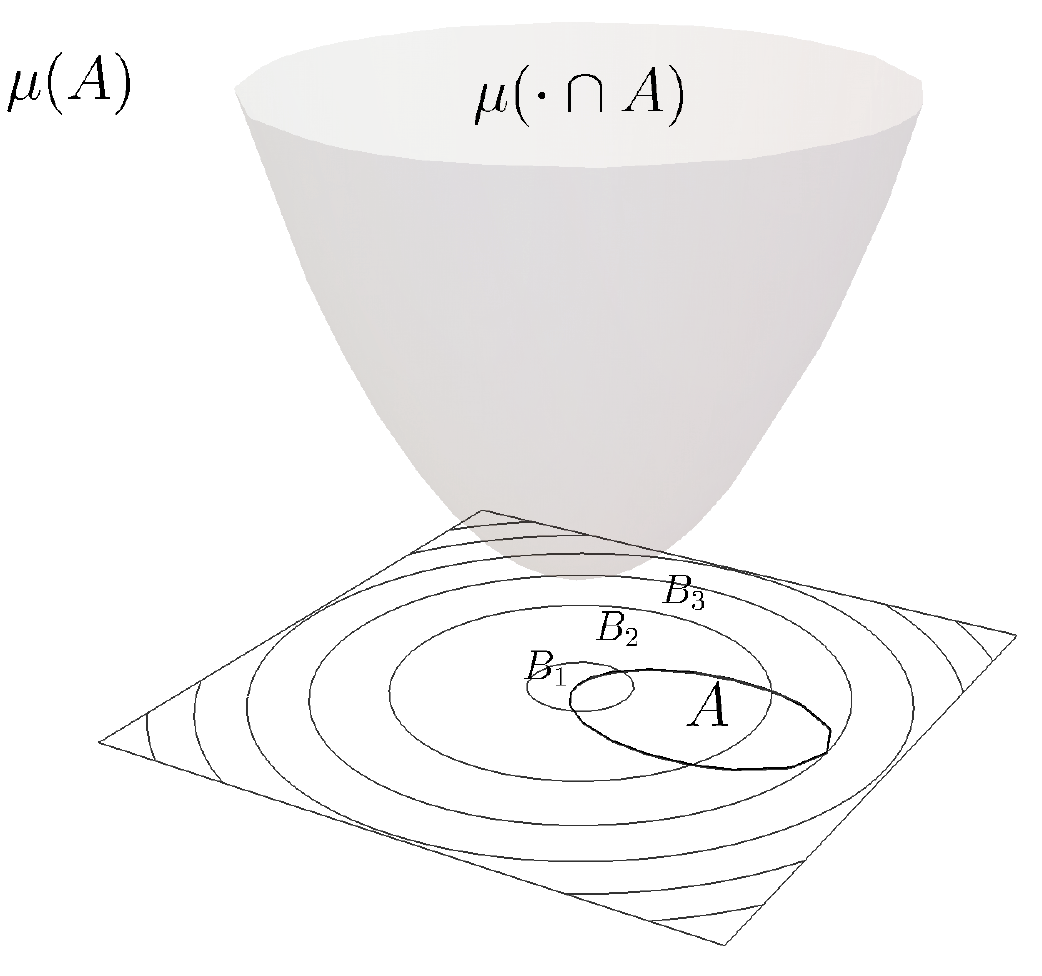}

}

\caption{$\mu\left(\cdot\cap A\right)$}

\end{figure}

\begin{flushleft}
\textbf{Stieltjes measures}
\par\end{flushleft}

It follows from \thmref{F1}, that elements $F$ in $\mathscr{H}\left(K_{\mu}\right)$
are signed measures on $\left(X,\mathscr{B}\right)$ s.t. $F\ll\mu$.
Set $T_{\mu}F:=\frac{dF}{d\mu}$, then 
\begin{equation}
T_{\mu}:\mathscr{H}\left(K_{\mu}\right)\xrightarrow{\;\approx\;}L^{2}\left(\mu\right)\label{eq:Fa1}
\end{equation}
is an \emph{isometric isomophism}, i.e., 
\begin{flushleft}
\begin{equation}
\left\Vert dF/d\mu\right\Vert _{L^{2}\left(\mu\right)}=\left\Vert F\right\Vert _{\mathscr{H}\left(\mu\right)}.\label{eq:Fa2}
\end{equation}
\par\end{flushleft}

In the special case of $\left(X,\mathscr{B}\right)=\left(\mathbb{R},\mathscr{B}\right)$,
or $J$ an interval, e.g., $J=\left[0,1\right]$ or $J=\left[0,\infty\right]$,
the general conclusions (\ref{eq:Fa1})-(\ref{eq:Fa2}) simplify as
follows. 

All signed measures on $\left(\mathbb{R},\mathscr{B}\right)$ have
the form $F=df$ for a bounded variation function $f$ on $\left(X,\mathscr{B}\right)$,
i.e., 
\begin{equation}
\int\varphi dF=\int\varphi df\label{eq:Fa3}
\end{equation}
as a Stieltjes measure, where 
\begin{equation}
df\left(\left[x,y\right]\right)=f\left(y\right)-f\left(x\right)\label{eq:Fa4}
\end{equation}
for intervals and extend to all $B\in\mathscr{B}$. Moreover the operator
$T_{\mu}f=f^{\left(\mu\right)}$ from before then agree (\ref{eq:Fa1})-(\ref{eq:Fa2}),
i.e., 
\begin{equation}
T_{\mu}f=\frac{df}{d\mu}\label{eq:Fa5}
\end{equation}
since 
\begin{equation}
df\in\mathscr{H}\left(K_{\mu}\right)\Longleftrightarrow df\ll\mu,\label{eq:Fa6}
\end{equation}
so that the Radon-Nikodym derivative $\frac{df}{d\mu}$ in (\ref{eq:Fa5})
is well defined and (see (\ref{eq:Fa4}))
\begin{equation}
f\left(y\right)-f\left(x\right)=\int_{X}\left(T_{\mu}f\right)d\mu,\label{eq:Fa7}
\end{equation}
which is the form we used before when $T_{\mu}f=f^{\left(\mu\right)}$.
And (\ref{eq:Fa7})$\Longrightarrow$ 
\begin{equation}
df\left(B\right)=\int_{B}f^{\left(\mu\right)}d\mu\label{eq:Fa8}
\end{equation}
where $df$ is the Stieltjes measure, $B\in\mathscr{B}$, and $\mathscr{B}$
a $\sigma$-algebra.

So we have $T_{\mu}f:=f^{\left(\mu\right)}$, and (\ref{eq:Fa8})
is the standard extension of $df$ defined first on intervals $df\left(\left[x,y\right]\right)=f\left(y\right)-f\left(x\right)$,
and then extended to Borel sets $df\left(B\right)$. Note that every
$\mu$ is a Stieltjes measure $f\left(x\right):=\mu\left([0,x]\right)$,
so of the form $\mu=df$. 

\section{\label{sec:class}A Hilbert space of equivalence classes}

In the earlier literature, authors typically only focus their analysis
on a fixed positive Borel measure $\mu$. This $\mu$ might be compared
to Lebesgue measure $\lambda$. When Stieltjes measures $df$ are
considered, it will then be relative to just this one measure $\mu$.
So when discussing $\nabla^{\left(\mu\right)}f=f^{\left(\mu\right)}$,
then consideration of the equation $df=f^{\left(\mu\right)}d\mu$
is really only picking out one component of $df$. Recall that the
family of Stieltjes measures $df$ account for all Borel measures.
And, in general, a Stieltjes measure will contain other non-zero components.

The focus in section \ref{sec:class} is the following: When we apply
the Jordan decomposition to a fixed Stieltjes measure $df$, then
the part of $df$ that is singular w.r.t. $\lambda$ may contain multiple
components, chosen in such a way that each of these components is
mutually singular w.r.t. the others. The emphasis below is this: We
introduce a Hilbert space $\mathscr{H}_{\text{class}}$ of \textquotedblleft sigma
functions\textquotedblright . Starting with a Stieltjes measure $df$,
we may identify its mutually singular components with orthogonal \textquotedblleft pieces\textquotedblright{}
in the Hilbert space $\mathscr{H}_{\text{class}}$.

We shall consider pairs $\left(f,\mu\right)$ where $f$ is a locally-bounded
variation function, and $\mu$ is a positive non-atomic Borel measure.
Following \cite{MR0282379}, one checks that the $\sim$ as specified
below will be an equivalence relation on pairs: 
\begin{equation}
\left(f_{1},\mu_{1}\right)\sim\left(f_{2},\mu_{2}\right)\;\underset{\left(\text{Defn}\right)}{\text{iff}}\;\exists\,\nu\:\left(\text{positive Borel measure}\right)\label{eq:Fh1}
\end{equation}
such that $\mu_{i}\ll\nu$, and 
\begin{equation}
\frac{df_{1}}{d\mu_{1}}\sqrt{\frac{d\mu_{1}}{d\nu}}=\frac{df_{2}}{d\mu_{2}}\sqrt{\frac{d\mu_{2}}{d\nu}}\quad\text{a.e. \ensuremath{\nu}.}\label{eq:Fh2}
\end{equation}
Moreover the set of such equivalence classes will form a Hilbert space,
with 
\begin{equation}
\left\Vert \text{class}\left(f,\mu\right)\right\Vert _{\mathscr{H}_{\text{class}}}^{2}=\int\left|\frac{df}{d\mu}\right|^{2}d\mu.\label{eq:Fh3}
\end{equation}
One further checks that orthogonality in the inner product in $\mathscr{H}_{\text{class}}$
happens precisely for classes with measures $\mu_{1}$ and $\mu_{2}$
which are mutually singular. 
\begin{thm}
If $f$ is given, locally of bounded variation. In addition, assume
that the sum in (\ref{eq:Fh4}) is finite, so the Stieltjes measure
$df$ is in $\mathscr{H}_{\text{class}}$. It follows that $\mathscr{H}_{\text{class}}$
induces a Hilbert norm as follows: 
\begin{equation}
\left\Vert f\right\Vert _{\mathscr{H}_{\text{class}}}^{2}=\sum_{\mu}\int\left|\frac{df}{d\mu}\right|^{2}d\mu\label{eq:Fh4}
\end{equation}
where the sum on the RHS in (\ref{eq:Fh4}) is over all $\mu$ s.t.
$df|_{supp\left(\mu\right)}\ll\mu|_{supp\left(\mu\right)}$, and distinct
terms in the sum correspond to mutually singular measures $\mu$. 
\end{thm}

\begin{proof}
The result is immediate from the discussion above, and \cite[ch 6]{MR0282379},
commutative multiplicity theory. When the function $f$ is fixed,
one checks from the definition of the equivalence relation (\ref{eq:Fh1})--(\ref{eq:Fh2}),
and an easy calculation, that each of the individual terms on the
RHS in (\ref{eq:Fh4}) in the sum-expression only depends on the equivalence
class determined by the measures $\mu$ entering into the summation.
(Note that the Hilbert space $\mathscr{H}_{\text{class}}$ of equivalence
classes is also called the Hilbert space of sigma-functions.)
\end{proof}
\begin{rem}
Nelson's sigma Hilbert space \cite{MR0282379} serves as a tool allowing
us to make precise the formal assertion for Stieltjes measures: 
\begin{equation}
df=\sum_{\mu}\left(\nabla^{\left(\mu\right)}f\right)d\mu\label{eq:Fh5}
\end{equation}
where the measures $\mu$ in (\ref{eq:Fh5}) are specified as in (\ref{eq:Fh4}).
Hence (\ref{eq:Fh5}) is justified when the function $f$ (in (\ref{eq:Fh5}))
yields a finite sum for the RHS in (\ref{eq:Fh4}). 

For the Stieltjes measure $df$, we therefore get the following evaluation
formula: For all $B\in\mathscr{B}_{1}$ (= the Borel $\sigma$-algebra),
we have 
\begin{equation}
df\left(B\right)=\sum_{\mu\in\mathscr{M}_{+}\left(df\right)}\int_{B\cap\text{esssup\ensuremath{\left(\mu\right)}}}\left(\nabla^{\left(\mu\right)}f\right)d\mu,
\end{equation}
where ``$\text{esssup}$'' refers to essential support. 
\end{rem}

The following example illustrates that there are choices of functions
$f$ for which the sum might be infinite.
\begin{example}
Consider $J=\left[0,1\right]$. Let $\mu_{n}=$ Lebesgue measure restricted
to $\left[\frac{1}{2^{n+1}},\frac{1}{2^{n}}\right]$, $n\in\mathbb{N}$.
Let $f\left(x\right)=\sqrt{x}$ over $\left[0,1\right]$, so that
$f'\left(s\right)=\frac{1}{2\sqrt{x}}$. Then, 
\begin{align*}
\left\Vert f\right\Vert _{\mathscr{H}_{\text{class}}}^{2} & \geq\sum_{n=1}^{\infty}\int\left|\frac{df}{d\mu_{n}}\right|^{2}d\mu_{n}\\
 & =\frac{1}{4}\sum_{1}^{\infty}\int_{\frac{1}{2^{n+1}}}^{\frac{1}{2^{n}}}\frac{1}{x}dx\\
 & =\frac{1}{4}\sum_{1}^{\infty}\left(\ln2\right)=\infty.
\end{align*}
\end{example}

\bibliographystyle{amsalpha}
\bibliography{ref}

\end{document}